\documentclass[11pt]{article}
\usepackage{latexsym,amsmath,amscd,amssymb,graphics,mathrsfs}
\usepackage{enumerate}
\usepackage{graphicx}
\usepackage{color}
\usepackage[colorlinks]{hyperref}
\usepackage{url}
\usepackage[all]{xy}

\newcommand{\rem}[1]{}
\makeatletter

\newcommand \al{\alpha}
\newcommand\be{\beta}
\newcommand\ga{\gamma}
\newcommand\de{\delta}
\newcommand\ep{\varepsilon}
\newcommand\ze{\zeta}
\newcommand\et{\eta}
\renewcommand\th{\theta}

\newcommand\ka{\kappa}
\newcommand\la{\lambda}
\newcommand\rh{\rho}

\newcommand\si{\sigma}

\newcommand\ph{\varphi}

\newcommand\ps{\psi}
\newcommand\om{\omega}
\newcommand\Ga{\Gamma}

\newcommand\Th{\Theta}

\newcommand\Si{\Sigma}

\newcommand\Ps{\Psi}
\newcommand\Om{\Omega}

\newcommand\ie{i.e.\ }

\newcommand\oo{{\infty}}

\renewcommand\o{\circ}
\renewcommand\div{\on{div}}
\newcommand\x{\times}
\newcommand\on{\operatorname}

\newcommand\Emb{\on{Emb}}

\newcommand\ex{\on{ex}}

\newcommand\vol{\on{c,vol}}

\newcommand\Ker{\on{Ker}}

\renewcommand\P{\mathcal{P}}
\newcommand\Diff{\on{Diff}}
\newcommand\Per{\on{Per}}

\newcommand\Vol{\on{Length}}

\newcommand\KKS{\on{KKS}}

\newcommand\SE{\on{SE}}
\newcommand\Zero{\on{Zero}}

\newcommand\g{\mathfrak g}
\newcommand\pa{\partial}
\newcommand\h{\mathfrak h}
\newcommand\Gr{\on{Gr}}

\newcommand\ZZ{\mathbb Z}
\newcommand\TT{\mathbb T}
\newcommand\RR{\mathbb{R}}

\newcommand\X{\mathfrak X}

\newcommand\R{\mathcal R}
\renewcommand\O{\mathcal O}

\newenvironment{proof}[1][Proof]{\noindent\textbf{#1.} }{\ \rule{0.5em}{0.5em}}

\def\XXint#1#2#3{{\setbox0=\hbox{$#1{#2#3}{\int}$ }
\vcenter{\hbox{$#2#3$ }}\kern-.5\wd0}}

\definecolor{burgund}{RGB}{153,0,51}      %burgundy

\begin{document}

\newtheorem{theorem}{Theorem}[section]
\newtheorem{definition}[theorem]{Definition}
\newtheorem{lemma}[theorem]{Lemma}
\newtheorem{remark}[theorem]{Remark}
\newtheorem{proposition}[theorem]{Proposition}
\newtheorem{corollary}[theorem]{Corollary}
\newtheorem{example}[theorem]{Example}

%\def\below#1#2{\mathrel{\mathop{#1}\limits_{#2}}}

%%%%%%%%%%%%%%%%%%%%%%%%%%%%%%%%%%%%%%%%%%%%%%
%%%%%%%%%%%%%%%%%%%%%%%%%%%%%%%%%
%%%%%%%%

\title{Vortex sheets in ideal 3D fluids, coadjoint orbits,\\ and characters}
\author{Fran\c{c}ois Gay-Balmaz$^{1}$ and Cornelia Vizman$^{2}$ }

\addtocounter{footnote}{1}
\footnotetext{LMD, \'Ecole Normale Sup\'erieure/CNRS, Paris, France.
\texttt{gaybalma@lmd.ens.fr }
\addtocounter{footnote}{1}}

\footnotetext{Department of Mathematics,
West University of Timi\c soara,
Romania.
\texttt{cornelia.vizman@e-uvt.ro}
\addtocounter{footnote}{1} }

\date{ }
\maketitle
\makeatother
%\begin{center} DRAFT \end{center}
%\maketitle

%|||-------------------text width----------------------|||



\begin{abstract} 
We describe the coadjoint orbits of the group of volume preserving diffeomorphisms of $\RR^3$ associated to the motion of closed vortex sheets in ideal 3D fluids. We show that these coadjoint orbits can be identified with nonlinear Grassmannians of compact surfaces  enclosing a given volume and endowed with a closed 1-form describing the vorticity density. If the vorticity density has discrete period group and is nonvanishing, the vortex sheet is given by a surface of genus one fibered by its vortex lines over a circle.
We determine the Hamilton equations for such vortex sheets relative to the Hamiltonian function suggested in \cite{khesin} and prove that there are no stationary solutions having rotational symmetries.
These coadjoint orbits are shown to be prequantizable if the period group of the 1-form and the volume enclosed by the surface satisfy an Onsager-Feynman relation, as argued in \cite{goldin2} for the case of open vortex sheets (tubes/ribbons).
We find a character for the prequantizable coadjoint orbits, 
as well as a polarization group on which the character extends, 
which is a first step beyond prequantization. 
\end{abstract}
%\tableofcontents

\section{Introduction}

Coadjoint orbits of volume preserving diffeomorphism groups are the natural phase spaces for regular and singular solutions of the Euler equations of an ideal fluid  \cite{Marsden-Weinstein}. 
In this paper we consider the coadjoint orbits associated to the motion of closed vortex sheet singular solutions of the 3D Euler equations, i.e., vorticities supported by closed surfaces in $ \mathbb{R} ^3$. Such coadjoint orbits are described as a certain class of nonlinear Grassmannians of vortex sheets indexed by the topological type of support, the type of vorticity density, and the enclosed volume. This identification is made via the momentum map associated to the action of the volume preserving diffeomorphism group on manifolds of vortex sheets.
In this realisation of the coadjoint orbits, the Kirillov-Kostant-Souriau symplectic form takes a particularly simple expression reminiscent of the expression of a canonical symplectic form.

If the closed one form describing the vorticity density has a discrete period group and is nonvanishing, the vortex sheet is shown to be given by a surface of genus one fibered by its vortex lines over a circle. For the Hamiltonian suggested in \cite{khesin}, given by the total length of the vortex lines, we determine the Hamilton equations of evolution of a vortex sheet $( \Sigma (t), \beta _ \Sigma(t)) $ to be
\begin{equation}\label{vse}
\dot \Sigma = kB ^{\perp_\Sigma} , \qquad \dot \beta _ \Sigma = \pounds_{(kB^{\top_\Sigma})}  \beta _ \Sigma.
\end{equation}
Here the surface $ \Sigma$ denotes the support of the vortex sheet while the closed 1-form $ \beta _ \Sigma$ describes the vorticity density. Equations \eqref{vse} involve the normal and tangential components of the binormal curvature $kB$, with respect to the support $ \Sigma $. The first equation is reminiscent of the vortex filament equation written as a binormal curvature flow.

These equations preserve the subset of circle-invariant vortex sheets (hence supported on surfaces of revolution), on which it is shown that there are no stationary solutions. 
%Connected components of these invariant submanifolds are described and t
The expression of the Hamilton equations, as partial differential equations for the parametrization of the vortex sheet, is given on the connected component of circle-invariant vortex sheets fibered by parallel circles (vortex lines).

If a genus one vortex sheet $( \Sigma , \beta _ \Sigma )$ satisfies the Onsager-Feynman relation, written as
\[
a \ell \in 2 \pi \mathbb{Z}  
\]
with $a$ the volume enclosed by $ \Sigma $ and $\ell$ the smallest period of $ \beta _ \Sigma $, we show that the corresponding coadjoint orbit is prequantizable, admits a polarization group, as well as a character that extends to that group. In particular we recover and rigorously justify 
several steps towards the quantization process of vortex sheets presented in \cite{goldin2} for the case of open vortex sheets (tubes/ribbons).

\medskip

Descriptions of coadjoint orbits of diffeomorphism groups in terms of nonlinear Grassmannians have been carrried out for several situations.
For the group of Hamiltonian diffeomorphisms, two classes of coadjoints orbits were described in \cite{Weinstein90}, \cite{Lee2009} and \cite{Haller-Vizman}. %\todo{We completed this class in \cite{GBV2018}, right?}
These classes were obtained via symplectic reduction for the dual pair of momentum maps associated to the Euler equations in \cite{GBV2018} and were generalized to coadjoint orbits of nonlinear flag manifolds in \cite{HV20}.

For the group of volume preserving diffeomorphisms, coadjoint orbits associated to codimension two singular solutions of the Euler equations (e.g.~vortex filaments in 3D) were described as nonlinear Grassmannians in \cite{Marsden-Weinstein} and \cite{Haller-Vizman}. The Hamiltonian evolution of a vortex filament in 3D space is given by the binormal curvature equation (e.g. \cite[Section~3.5]{Khesin-Wendt}). It has been extended to higher dimensional codimension two submanifolds (vortex membranes) in \cite{HV03} and \cite{khesin}, given by a skew mean curvature equation, with respect to a regularized Hamiltonian taken as the volume of the membrane.
Coadjoint orbits associated to codimension one singular solutions were considered in \cite{goldin2} for open vortex sheets in 3D ideal fluids. A Hamiltonian was suggested in \cite{khesin} for vortex sheets, given by the total length of the vortex lines. The motion of vortex sheet was given a geodesic description in \cite{Lo2012} and has been recently formulated within a groupoid framework in \cite{IK}.

\medskip 

The paper is structured as follows. In section \ref{sec_2}, coadjoint orbits of vortex sheets are described as certain classes of  nonlinear Grassmannians and the KKS symplectic form is derived in this description. In section \ref{Sec_3} we focus on the Hamilton equations for vortex sheets with discrete period group and nonvanishing vorticity density. We also consider the invariant submanifold of circle-invariant vortex sheets and show that there are no stationary vortex sheets of this type.
In section \ref{Sec_4} it is shown that a coadjoint orbit satisfying the Onsager-Feynman condition is prequantizable, with prequantum bundle explicitly constructed as the quotient of a certain decorated nonlinear Grassmanian by a discrete group. In section \ref{cara}, a polarization group is given and a character is geometrically constructed that integrates the vortex sheet momenta on the polarization group.

We project to explore the higher dimensional version of these results in a future work.

\section{Coadjoint orbits of vortex sheets}\label{sec_2}

In this section, we describe coadjoint orbits of the group  of compactly supported volume preserving diffeomorphisms, $\Diff_{\vol}(\RR^3)$, associated to the motion of vortex surfaces in ideal 3D fluids,
also called vortex sheets. We also determine the  KKS symplectic form on these orbits.
Coadjoint orbits for codimension two vortex membranes (e.g.~vortex filaments in 3D) were studied in \cite{Marsden-Weinstein}, \cite{Haller-Vizman}. 

Let $S$ be a compact oriented surface
and let $\operatorname{Diff}_+(S)$ denote the group of  orientation preserving diffeomorphisms of $S$. 
Let $\be\in\Om^1(S)$ be a closed 1-form with isolated zeros. By compactness, the zero set $\Zero(\be)$ is finite.
The surface $S$ describes the support of the singular vorticity distribution, while the 1-form $ \beta $ describes its direction and strength.

\subsection{Principal bundles of embeddings}

Consider the Fr\'echet manifold $\operatorname{Emb}(S,\RR^3)$ of embeddings of $S$ into $\RR^3$. It is the total space of the principal $ \operatorname{Diff}_+(S)$ bundle
\begin{equation}\label{prinz}
\pi : \operatorname{Emb}(S,\RR^3) \rightarrow\operatorname{Gr}^S, \quad f \mapsto \Si=f(S) ,
\end{equation}
where $\operatorname{Gr}^S$ denotes the {\it nonlinear Grassmannian} of all oriented surfaces in $\RR^3$ of type $S$ (including the orientations). 
By the transitivity of the action of the Lie algebra of compactly supported divergence free vector fields $ \X_c(\RR^3)$ on $\operatorname{Emb}(S, \RR^3)$, see \cite[Chapter~8]{H76},  the tangent space at $f$ can be written
as $T_f \operatorname{Emb}(S,\RR^3)=\left \{ u \circ f: S \rightarrow T\RR^3 : u \in \mathfrak{X} (\RR^3)\right \}$. 
The tangent space at $\Si$ to the nonlinear Grassmannian
can be identified with the space of sections of the normal bundle $T\Si^\perp$ 
with respect to the Euclidean metric on $\RR^3$:
\begin{equation}\label{Coo}
T_\Si\Gr^S=\Ga(T\Si^{\perp}) \simeq C^\oo(\Si).
\end{equation}
With $n$ denoting the unique unit normal vector field to $\Si$ compatible with the orientations of $\Si$ and of $\RR^3$, the second identification becomes $u_\Si\leftrightarrow u_\Si\cdot n$,
for any normal vector  field $u_\Si$ on $\Si$.
We define the \textit{decorated nonlinear Grassmannian}  $\Gr^{S,\be}$
as a bundle over $\Gr^S$ associated to the principal bundle \eqref{prinz}, namely:
\begin{equation}\label{grom1}
\Gr^{S,\be}:=\Emb(S,\RR^3)\x_{\Diff_+(S)}\O_\be,
\end{equation}
with $\O_\be\subseteq\Om^1(S)$ denoting the $\Diff_+(S)$ orbit of $\be$.
The map 
\begin{equation}\label{phal}
[f,\al]\mapsto \left(f(S),f_*\al\right)=(\Si,\be_\Si),\quad f\in\Emb(S,\RR^3),\ \al\in\O_\be,
\end{equation} 
makes the identification of \eqref{grom1} with the space of oriented surfaces in $\RR^3$ of type $S$, each of them endowed (decorated) with a 1-form of type $ \beta $
\[
\operatorname{Gr}^{S, \beta}=\left \{(\Si, \beta _\Si) : \Si \in \operatorname{Gr}^S,\; \beta _\Si \in \Omega ^1 (\Si)\;\; \text{s.t.} \;\exists\;\Psi \in \operatorname{Diff}_+(S,\Si),\; \Psi ^\ast\beta _\Si =  \beta \right \}.
\]
Note that $ \beta _\Si$  is closed with finite zero set, just like $\be$. 
The associated bundle projection becomes the \textit{forgetful map}
$$
\Gr^{S,\be}\to\Gr^S,\quad (\Si,\be_\Si)\mapsto \Si.
$$
Moreover, formally,
\begin{equation}\label{pib1}
\pi ^\beta : \operatorname{Emb}(S,\RR^3) \rightarrow\operatorname{Gr}^{S, \beta }, \quad f \mapsto (f(S), f _\ast\beta ) =(\Si, \beta _\Si )
\end{equation}
is a principal bundle with structure group $ \operatorname{Diff}_+(S, \beta )$, the group of $\be$-preserving diffeomorphisms that preserve the orientation. Note that it is not a Lie group in general, 
still it has an associated Lie algebra of  $\be$-preserving vector fields 
\begin{equation}\label{xsb}
\X(S,\be)
=\{v\in\X(S):\be(v)=\text{constant}\}.
\end{equation}

\paragraph{Splitting of the tangent space.} The Euclidean metric on $\RR^3$ naturally endows the principal bundle \eqref{prinz} with a principal connection:
\begin{equation}\label{conn}
\om\in\Om^1\big(\Emb(S,\RR^3),\X(S)\big),\quad \om(u\o f)=f^*\big(u|_{ \Sigma }^\top\big),
\quad u\in\X(\RR^3).
\end{equation}
with $ \Sigma = f(S)$. The horizontal lift of $u_ \Sigma \in T_ \Sigma \Gr^{S,\beta }$ at $f \in \Emb(S,\RR^3)$ is $\left( (u_ \Sigma \cdot n) n \right)  \circ f \in T_f \Emb(S,\RR^3)$.
This connection, on its own turn, induces a connection on the associated bundle 
$\operatorname{Gr}^{S, \beta }$.
This splits in a natural way the tangent space to the decorated Grassmannian 
into its horizontal and vertical parts
(similarly to the weighted nonlinear Grassmannians addressed in \cite{GBV2018}).
The tangent space to $\O_\be$, the $\Diff_+(S)$ orbit of $\be$,
is  $\{\pounds_v\be:v\in\X(S)\}= d (\beta( \mathfrak{X}  (S)))
=d C_\be^\oo(S)$, with $C^\oo_{\be}(S)$ denoting the space of smooth functions vanishing on the finite subset $\Zero(\be)$.
This yields the horizontal/vertical decomposition:
\begin{equation}\label{split}
T_{(\Si,\beta _\Si)} \operatorname{Gr}^{S,\beta }
=T_\Si\Gr^S \times d \left(\beta _\Si( \mathfrak{X}  (\Si))\right)
\stackrel{\eqref{Coo}}{=}C^\oo(\Si) \times d C^\oo_{\be_\Si}(\Si).
\end{equation}

The bundle projection $\pi^\be$ in \eqref{pib1} being $\Diff(\RR^3)$ equivariant
for the natural action on the decorated Grassmannian, $\ph\cdot(\Si,\be_\Si)=(\ph(\Si),\ph_*\be_\Si)$, the infinitesimal generators $\ze_u^{\Emb}$ and $\ze_u^{\Gr}$ for $u\in\X(\RR^3)$ are $\pi^\be$ related.
Thus the tangent map to $\pi^\beta $ can be written 
in the splitting \eqref{split} as
\begin{equation}\label{T_pi_beta} 
T_f \pi^\be(\ze_u^{\Emb})=T_f \pi ^\beta (u \circ f)= \big( u|_\Si^\perp, { \pounds _{ u|_\Si^\top}\beta _\Si}\big)=(u|_\Si\cdot n, d  (  i  _{u|^\top_\Si} \beta _\Si))=:\ze^{\Gr}_u(\Si,\be_\Si)
\end{equation}
where $u|_\Si=u|_\Si^\top+u|_\Si^\perp$ is the orthogonal decomposition along $\Si=f(S)\subseteq\RR^3$.

%%%%%%%%%%%%

\subsection{Enclosing a constant volume}

Let $\mu$ be the volume form on $\RR^3$ induced by the Euclidean metric.
Given $a>0$, we define the manifold
\begin{equation}\label{emba1}
\operatorname{Emb}_a(S,\RR^3) := \Big\{ f \in\operatorname{Emb}(S,\RR^3) : \int_{S}f ^\ast  \nu =a \Big\},
\end{equation}
where  $\nu\in\Om^2(\RR^3)$ is such that $ \mu =  d  \nu $. One notices that $\operatorname{Emb}_a(S,\RR^3)$ is independent of the choice of $\nu$.
By restriction of the principal bundle \eqref{prinz},  we get another 
principal $ \operatorname{Diff}_+(S)$ bundle
$\pi : \operatorname{Emb}_a(S,\RR^3) \rightarrow\operatorname{Gr}^S_a$,
this time over  the manifold
\[
\operatorname{Gr}^S_a:=\Big\{\Si\in\Gr^S: \int_\Si\nu=a\Big\}
\]
of oriented surfaces in $\RR^3$ that enclose a constant volume $a$.

The tangent spaces are: 
\begin{equation}\label{emba}
T_f \operatorname{Emb}_a(S,\RR^3)=\Big \{ u \circ f: S \rightarrow T\RR^3 : u \in \mathfrak{X}  (\RR^3),\; \int_S f ^\ast  i  _u \mu=0\Big \}
\end{equation}
and
$T_\Si \operatorname{Gr}^S_a=\left \{ u_\Si \in \Gamma (T\Si^\perp) : \int_\Si i  _{u_\Si} \mu =0\right \}$.
Using  the volume form induced on the surface $\Si$ by the Euclidean metric, 
namely  $\mu_\Si= i _{n}\mu$,
the identification in \eqref{Coo} yields, by restriction,
\begin{equation}\label{tngr}
T_\Si \operatorname{Gr}_a^S\simeq C_0^\oo(\Si):=\Big\{f\in C^\oo(\Si): \int_\Si f\mu_\Si=0\Big\}.
\end{equation}
Indeed, $\int_\Si i _{u_\Si}\mu
=\int_\Si (u_\Si\cdot n) i _n\mu=\int_\Si f\mu_\Si$, for all $u_\Si\in\Ga(T\Si^\perp)$.

We consider also the \textit{decorated Grassmannian enclosing a constant volume}
\begin{equation}\label{Gr_a_S_beta} 
\operatorname{Gr} _a^{S, \beta}:=\left \{(\Si, \beta _\Si) \in\Gr^{S,\be} : \Si \in \operatorname{Gr}^S_a\right\},
\end{equation} 
together with the (formal) principal $\operatorname{Diff}_+(S, \beta )$ bundle,
obtained from \eqref{pib1},
\begin{equation}\label{pib2}
\pi ^\beta : \operatorname{Emb}_a(S,\RR^3) \rightarrow\operatorname{Gr}^{S, \beta }_a.
\end{equation}
This time the projection is $\Diff_{\rm vol}(\RR^3)$ equivariant.
From now on, we will refer to the decorated Grassmannian \eqref{Gr_a_S_beta} formally as a \textit{manifold of vortex sheets of type $(S, \beta )$}.

\subsection{Symplectic form on the manifold of vortex sheets}

The 1-form $\be$ on $S$ and the volume form $\mu=d\nu$
on $\RR^3$ jointly define, in a natural way, a 2-form $\om$ on  $\Emb(S,\RR^3)$:
\begin{equation}\label{tinta}
\om(u_1\o f,u_2\o f)=\int_S f^* i _{u_2} i _{u_1}\mu\wedge\be,\quad u_1,u_2\in\X(\RR^3).
\end{equation}
In the hat calculus notation \cite{Vizman}, this form  is written as
\begin{equation}\label{omhat}
\om=\widehat{\mu\cdot\be}=d(\widehat{\nu\cdot\be)}. 
\end{equation}
It is invariant under the commuting actions
of the reparametrization group $\Diff_+(S,\be)$  and of the ambient space
diffeomorphism group  $\Diff_{\rm vol}(\RR^3)$, with infinitesimal generators 
$\zeta_v(f):= Tf\circ v$ and $\zeta_u(f):= u\o f$ respectively.

\begin{proposition}\label{simp}
The restriction of $\om$ to  $\Emb_a(S,\RR^3)$
descends via \eqref{pib2} to a symplectic form 
$\Om$ on $\operatorname{Gr}^{S, \beta }_a$:
\begin{equation}\label{pibeom}
(\pi^\be)^*\Om=\om.
\end{equation}
\end{proposition}

\begin{proof}
We know that $\om$ in \eqref{tinta} is $\Diff_+(S,\be)$-invariant.
We prove below that the kernel of $\om$ is generated by the $\X(S,\be)$ action.
It follows that the closed 2-form $\om$ on $\Emb_a(S,\RR^3)$ is  $\pi^\beta$-basic, 
thus it  descends to a closed non-degenerate 2-form $\Om$ on the manifold of vortex sheets $\operatorname{Gr} _a^{S, \beta}$. In general $\Om$ is not an exact form, even though its pullback $\om$ is exact.

Notice that all the infinitesimal generators $\zeta_w(f):=Tf\o w$ with $w\in\X(S)$
are tangent to $\Emb_a(S,\RR^3)$ at $f$.
We choose an arbitrary element $u_f:=u\o f\in T_f\Emb_a(S,\RR^3)$ in the kernel of $\om$.
Besides the condition $\int_Sf^* i _{u}\mu=0$ from \eqref{emba}, we also get
$$
0=\om(u_ f,Tf\o w)=-\int_S\be(w)f^* i _{u}\mu,\quad \text{ for all }w\in\X(S).$$
The 1-form $\be$ having a finite number of zeros, we deduce that
$f^* i _{u}\mu=0$, which implies that $u\in\X(\RR^3)$ is tangent to the surface $f(S)$.
This ensures the existence of $v\in\X(S)$ such that $u_ f=Tf\o v$.
It belongs to the kernel of $\om$, hence $0=\om(Tf\o v,u'\o f)=\int_S \be(v)f^* i _{u'}\mu$ for all $u'\in\X(\RR^3)$ such that $\int_Sf^* i _{u'}\mu=0$.
We get that the function $\be(v)$ must be constant,
so, by \eqref{xsb}, the vector field $v\in\X(S,\be)$ and $u_ f=\ze_v(f)$. 
\end{proof}

\medskip

Since the zero set of $\be_\Si$ is finite, the nondegenerate pairing 
$$
C_0^\oo(\Si)\x d  C^\oo_{\be_\Si}(\Si) \to\RR,\quad
(\rh,d\la)=\int_\Si \rh\la \mu_\Si
$$
permits to write the tangent space to $\Gr_a^{S,\be}$
%to identify $ d  C^\oo_{\be_\Si}(\Si)$ with the dual of $C_0^\oo(\Si)$.
 as
\begin{equation}\label{split_a}
T_{(\Si,\beta _\Si)} \operatorname{Gr}^{S,\beta }_a=T_\Si\Gr_a^S \times  d  \left(\beta _\Si( \mathfrak{X}  (\Si))\right)=C_0^\oo(\Si) \times
C_0^\oo(\Si)^*,
\end{equation}
using also \eqref{tngr} and \eqref{split}.
That these are canonical coordinates for the symplectic form $\Om$ is shown in the next proposition.

\begin{proposition}
The symplectic form  $\Om$  on 
$\operatorname{Gr} _a^{S, \beta}$,  given in Proposition \ref{simp},
becomes
\begin{equation}\label{cano}
\Omega _{(\Si,\be_\Si)}\left( (\rh_1,d\la_1),(\rh_2,d\la_2)\right) =
-\int_\Si(\rh_1\la_2 -\rh_2\la_1 )\mu_\Si
\end{equation}
in the description \eqref{split_a} of the tangent space.
\end{proposition}

\begin{proof}
First we notice that the expression of $\Om$ does not depend on the choice of the functions $ \lambda_1 $ and $ \la_2$ in $C^\oo_{\be_\Si}(\Si)$, since $\rh_1\mu_\Si$ and $\rh_2\mu_\Si$ are exact forms on $\Si$.

Because of the infinitesimal transitivity result in Lemma \ref{toto},
it is enough to check the identity on infinitesimal generators
$\ze^{\Gr}_{u_i}(\Si,\be_\Si)=(\rh_i, d\la_i)$,
for $u_i\in\X_{\vol}(\RR^3)$, \ie $ u_i|_\Si^ \perp =\rh_in$ and 
$ i  _{u_i|_\Si^ \top } \beta _\Si=\la_i$, for  $i=1,2$.
Hence
\begin{multline*}
\Om_{(\Si, \beta _\Si)} \left( (\rh_1, d  \lambda_1 ), (\rh_2,  d  \la_2)\right)
\stackrel{\eqref{tinta}}{=}\int_{\Si} ( i  _{ u_2} i  _{u_1 }\mu) 
\wedge \be_{\Si}  
=\int_{\Si} ( i  _{ u_2|_\Si^ \perp } i  _{ u_1|_\Si^ \top }\mu) 
\wedge \be_{\Si}  + \int_{\Si}  ( i  _{ u_2|_\Si^ \top } i  _{u_1|_\Si^ \perp }\mu)\wedge \be_{\Si}  \\
= \int_{\Si} ( i  _{u_1|_\Si^ \top } \be_{\Si})( i  _{ u_2|_\Si^ \perp }\mu )- \int_{\Si}  ( i  _{ u_2|_\Si^ \top }\be_{\Si})( i  _{ u_1|_\Si^ \perp }\mu) =-\int_\Si(\rh_1\la_2-\rh_2\la_1)\mu_\Si.
\end{multline*}
Only the mixed terms survive at step two:
$ i  _{ u_2|_\Si^ \perp } i  _{ u_1|_\Si^ \perp }\mu =0$
since all the normal vectors to $\Si$ are colinear
and  $   i_\Si^*( i  _{ u_2|_\Si^ \top } i  _{u_1|_\Si^ \top }\mu)  =0$ since $i_\Si^*\mu=0$.
\end{proof}

\subsection{Hamiltonian action on manifolds of vortex sheets}

The action of $\Diff_{\rm vol}(\RR^3)$
on the manifold of embeddings that enclose a constant volume, $\Emb_a(S,\RR^3)$, descends to an action on the manifold of vortex sheets
$\Gr_a^{S,\be}$:
\begin{equation}\label{act}
\ph\cdot(\Si,\be_\Si)=(\ph(\Si),\ph_*\be_\Si).
\end{equation}
Every divergence free vector field in $\X_{\rm vol}(\RR^3)$ admits a potential 1-form
$\al\in\Om^1(\RR^3)$, i.e.~$ i _u\mu=d\al$ , which we denote $u=X_ \alpha $. 

\begin{lemma}
The action of $\Diff_{\rm vol}(\RR^3)$ on the symplectic manifold
$(\Gr_a^{S,\be},\Om)$ is Hamiltonian with $\Diff_{\rm vol}(\RR^3)$ equivariant momentum map
\begin{equation}\label{mome}
J: \operatorname{Gr} _a^{S, \beta} \rightarrow\mathfrak{X} _{\rm vol}(\RR^3)^\ast , \quad\left\langle J(\Si,\beta _\Si), X_\alpha  \right\rangle =\int_\Si \al\wedge\beta _\Si,\quad\al\in\Om^1(\RR^3).
\end{equation}
\end{lemma}

\begin{proof}
The map $J$  doesn't depend on the choice of the potential 1-form $\al$.
Any two potential forms differ by a closed  1-form on $\RR^3$, hence by an exact one, and $\int_\Si d h\wedge\be_\Si=0$ for all $h\in C^\oo(\RR^3)$.

To show that $J$ is a momentum map, we check that for any $u=X_\al$, the infinitesimal generator $\ze_u^{\Gr}$ 
is a Hamiltonian vector field on $\Gr_a^{S,\be}$  with Hamiltonian 
function $H_u(\Si,\be_\Si):=\int_\Si\al\wedge\be_\Si$.
We use the hat calculus in \cite{Vizman}:
on one hand we have 
$$
(\pi^\be)^* i _{\ze_u^{\Gr}}\Om= i _{\ze_u^{\Emb}}\ \om
=\widehat{ i _u\mu\cdot\be}=\widehat{d\al\cdot\be}=d(\widehat{\al\cdot\be}),
$$
on the other hand, the pullback of the function $H_u$ by $\pi^\be$ is 
$$((\pi^\be)^*H_u)(f)=H_u(f(S),f_*\be)=\int_{f(S)}\al\wedge f_*\be
=\int_S f^*\al\wedge\be=(\widehat{\al\cdot\be})(f),$$ 
for all $f\in\Emb_a(S,\RR^3)$.
Knowing that for all $ \ph\in {\operatorname{Diff}_{\rm vol}(\RR^3 )}$ we have $ \operatorname{Ad}_{ { \varphi }} X_ \alpha = X_{ \varphi _\ast \alpha }$,
the equivariance  of the momentum map $J$ follows easily.
\end{proof}

\subsection{Manifolds of vortex sheets as coadjoint orbits}

In this section we show that the connected  components of the manifold of vortex sheets $\Gr_a^{S,\be}$
can be seen as  coadjoint orbits of $\Diff_{\vol}(\RR^3)$, the identity component of the group of compactly supported volume preserving diffeomorphisms of $\RR^3$.
For this we use the following well known fact (found for instance in
the appendix in \cite{HV20}):

\begin{proposition}\label{spcase}
Suppose the action of the Lie group $G$ on the symplectic manifold $(M,\Om)$ is transitive and infinitesimally transitive, with injective equivariant momentum map $J:M\to\g^*$.
Then $J$ is one-to-one onto a coadjoint orbit of $G$.
Moreover, it pulls back the Kostant--Kirillov--Souriau symplectic form $\omega_{\KKS}$ on the coadjoint orbit to the symplectic form $\Om$.
\end{proposition}

It is easy to see that the $\Diff_{\vol}(\RR^3)$ equivariant momentum map \eqref{mome}
is injective. Indeed, if $J(\Si_1, \beta _{1})= J(\Si_2, \beta _{2})$, then necessarily the surfaces coincide, $\Si_1=\Si_2$, because the 1-forms $ \beta _{1}$ and $ \beta _{2}$ have a finite number of zeros. We thus have 
$\be_2\in\Om^1(\Si_1)$ and $ \int_{\Si_1 }  \al\wedge (\beta _{1}- \beta _{2}) =0$, for all $\alpha\in\Om^1(\RR^3)$, hence $\beta _{1}= \beta _{2}$.

Next we  show that the action of the identity component of $\Diff_{\vol}(\RR^3)$
on connected  components of $\Gr_a^{S,\be}$ is transitive and
infinitesimally transitive.

\begin{lemma}\label{bastion} Let $N$ be an orientable codimension one submanifold of the $m$-dimensional manifold $M$ with inclusion $i_{N}:N\to M$.
Given   $ Y \in \mathfrak{X}  (N)$, there exists $ \si \in \Omega ^{m-2}(M)$ such that $i_{N} ^\ast \si =0$ and $ X_ \si |_N= Y$.
\end{lemma} 

\begin{proof} We endow $M$ with a Riemannian metric $g$
and we denote by $\mu$ the induced volume form. Since $N$ is orientable, $TN^{\perp}$ is a trivial line bundle, the trivialization being given by
the unit normal vector field $n$ along $N$, namely
 $TN^{\perp} \rightarrow N \times \mathbb{R}  $, $u_x\mapsto (x, u _x \cdot n)$. Thus, we can identify a tubular neighborhood $V$ of $N$ in $M$ with $N \times\, (-1,1)$.
With this identification, the volume form $ \mu $ on $V$ reads $dt \wedge \omega$, where $\om$ is a $t$-dependent $(m-1)$-form on $N$.

We define $\widetilde{Y} \in \mathfrak{X}  (V)$ by $\widetilde{Y}(x,t):=Y(x)+0 \partial _t $ on $V$.
We consider the form $ t  i  _{\widetilde{Y}}\omega \in \Omega ^{m-2}(V)$ and extend it to $ \si \in \Omega ^{m-2}(M)$. Let us check that $ \si $ satisfies the two required conditions.
On one hand, we have $i_{N} ^\ast \si = i_{N} ^\ast \left( i  _{tY}\omega  \right) =0$, since $N\subseteq V$ is characterized by $t=0$. 
On the other hand, on $V$  we have
$  i  _{X_ \si - \widetilde{Y}} \mu=  d  \si -   i  _{\widetilde{Y}} \mu =  d  \left( t  i  _{\widetilde{Y}} \omega \right) -  i  _{\widetilde{Y}}(dt \wedge \omega )=t  d   i  _{\widetilde{Y}} \omega $. We get $(X_\si-\widetilde Y)|_N=0$,
so $X_ \si |_N=Y$.
\end{proof}

\medskip

\begin{lemma}\label{toto}
The Lie algebra $\X_{\vol}(\RR^3)$ of compactly supported 
divergence free vector fields acts transitively  on the manifold $\Emb_a(S,\RR^3)$.
\end{lemma}

\begin{proof}
We consider $u\o f\in T_f \operatorname{Emb}_a(S,\RR^3)$ with $u \in \mathfrak{X}_c  (\RR^3)$, hence $\int_\Si i  _u \mu =0$, where $\Si=f(S)$. Therefore, there exists $ \al_\Si \in\Om^1(\Si)$ with $i_\Si ^\ast   i  _u \mu =  d \al_\Si $. 
We choose $ \alpha \in \Omega_c ^{1}(\RR^3)$ 
such that $ i_\Si^\ast \alpha = \alpha _\Si$. We have $i^*_\Si (  i  _{u- X_ \alpha } \mu )=  d  \alpha _\Si- i_\Si ^\ast  d  \alpha =0$, hence $Y:=(u- X_ \alpha)|_\Si \in \mathfrak{X}  (\Si)$.
By the Lemma \ref{bastion} one can find 
$ \si \in \Omega ^{1}(\RR^3)$ with $i_\Si ^\ast \si=0$ and such that $X_ \si |_\Si =Y$. Because $\Si$ is compact, $\si$ can be chosen with compact support.
Hence the vector field $X_{ \alpha + \si }\in\X_{\vol}(\RR^3)$
has the property $u|_\Si=X_{ \alpha + \si }|_\Si$, so $u\o f=X_{ \alpha + \si }\o f$, which shows the infinitesimal transitivity of the $\mathfrak{X}_{\vol}(\RR^3)$ action  on $\operatorname{Emb}_a(S,\RR^3)$.
\end{proof}

\medskip 

It is clear that the construction in the proof of Lemma \ref{toto} 
can be done smoothly depending on a parameter.
Since $\Diff_{\vol}(\RR^3)$ is locally connected by smooth arcs,
we conclude that the action is locally transitive.

\begin{lemma}
The identity component of $\Diff_{\vol}(\RR^3)$ acts transitively on 
connected  components of the manifold of vortex sheets $\Gr_a^{S,\be}$.
\end{lemma}

Now Proposition \ref{spcase} can be applied to the momentum map for vortex sheets.

\begin{theorem}\label{Pere_Tranquille} 
The restriction of the momentum map 
$J: \operatorname{Gr} _a^{S, \beta}\rightarrow\mathfrak{X} _{\vol}(\RR^3)^\ast$
in \eqref{mome} to any connected   component of the manifold of vortex sheets $\Gr_a^{S,\be}$ is one-to-one onto a coadjoint orbit of the identity component of  $\Diff_{\vol}(\RR^3)$, the Lie group of compactly supported volume preserving diffeomorphisms. 
The Kostant--Kirillov--Souriau symplectic form $\om_{\rm KKS}$ on the coadjoint orbit satisfies $J^*\om_{\rm KKS}=\Om$, for the symplectic form $\Om$ in \eqref{pibeom}.
\end{theorem} 

%%%%%%

\paragraph{Vector field approach to vorticity density.}
In \cite{goldin2}, open sheets (ribbons/tubes) in 3D are treated and a vector field $\ga_\Si$ tangent to the vortex lines of $\Si$ is used to describe the vorticity density (also called vortex sheet strength).

In our description of vortex sheets using closed 1-forms $\beta _ \Sigma \in \Omega ^1 ( \Sigma )$ for the vorticity density,  the vector field $ \gamma _ \Sigma $ on $ \Sigma $ is defined by $ i _{\ga_\Si}\mu_\Si=\be_\Si$ with $\mu_\Si= i _n\mu$ the volume form on the oriented surface $ \Sigma $ induced by the Euclidean metric. Note that $ \gamma _ \Sigma $ is a divergence free vector field with respect to $ \mu _ \Sigma $
and $\ga_\Si=n\x \be_\Si^\sharp$.

In terms of the vorticity density vector field $ \gamma _ \Sigma $, the momentum map \eqref{mome} becomes
\[
\langle J(\Si,\ga_\Si),X_\al\rangle=\int_\Si\al\wedge i _{\ga_\Si}\mu_\Si
=\int_\Si\al(\ga_\Si)\mu_\Si,\quad \al\in\Om_c^1(\RR^3).
\]
Let $\chi:=\al^\sharp\in\X_c(\RR^3)$ be the vector potential of $u=X_\al\in\X_{\vol}(\RR^3)$, obtained from the potential $\al\in\Om^1(\RR^3)$
by rising indices with the Euclidean metric, i.e, $ u = \operatorname{curl}\chi$.
Using $\chi$, the momentum map is given, as in \cite{goldin2}, by
\begin{equation}\label{gold}
\langle J(\Si,\ga_\Si),u\rangle=\int_\Si (\ga_\Si\cdot\chi)\mu_\Si.
\end{equation}
Note that the description of vorticity density with a closed 1-form $ \beta _ \Sigma $ is more natural since, unlike the divergence free vector field $ \gamma _ \Sigma $,  the 1-form is naturally transported via push-forward by diffeomorphisms in the coadjoint action, and hence is associated to a given closed 1-form $ \beta $ on $S$, see \eqref{pib1}.

%%%%%%%%%%%%%%%%%%\newpage

\section{Vortex sheets dynamics in 3D fluids}\label{Sec_3}

{In this section we characterize the vortex lines on vortex sheets $(\Si,\be_\Si)$, having discrete period group $\ell\ZZ$ and non-vanishing vorticity density, as the fibers of a certain fibration  $ \Sigma \rightarrow \TT_\ell$.} 
We then determine the Hamilton equations governing the motion of such vortex sheets. Denoting by $kB$ the binormal curvature to the vortex lines, these equations read
\begin{equation}\label{Ham_PDE} 
\left\{
\begin{array}{l}
\vspace{0.2cm}\dot \Sigma = k B^{\perp_{\Sigma}}\\
\dot \beta _ \Sigma = d  ( \beta _ \Sigma ( kB ^{\top_ \Sigma })).
\end{array}\right.
\end{equation} 
In particular, the deformation of the support of the vortex sheet is governed by the normal component of the binormal curvature of the vortex lines, while the deformation of the vorticity density is governed by the variation of the volume 
$(\gamma _ \Sigma \times kB) \cdot n$
generated by the vorticity density vector field, the binormal curvature, and the normal to the support. Finally, we focus  on the invariant
submanifold of circle-invariant vortex sheets, we show that there are no stationary solutions of this
type, and we determine the form of the Hamiltonian PDE \eqref{Ham_PDE} on the connected component of
circle-invariant vortex sheets fibered by parallel vortex lines.

\subsection{Vortex lines and their curvatures on vortex sheets}

\paragraph{Vortex lines as fibration over a circle.} Recall that the period group of a closed 1-form $\beta \in \Omega ^1 (S)$ is the subgroup of $ \mathbb{R}  $ defined by
\[
\operatorname{Per}(\beta) := \left\{ \int_ \gamma \beta : [\gamma ] \in H_1(S,\ZZ)\right\}.
\]

\begin{lemma}\label{lemma_31} 
If the period group of $\be$ is discrete, $\operatorname{Per}( \beta)= \ell \mathbb{Z}$ for some $\ell>0$, then there is a map
\begin{equation}\label{map_b} 
b: S \rightarrow \mathbb{R}/\ell\mathbb{Z}=:\TT_\ell ,
\end{equation} 
such that $ \beta = b ^\ast\vartheta_\ell $, where $ \vartheta_\ell \in \Omega ^1 (\TT_\ell)$ denotes the {angular form of $\mathbb{T}_\ell$, \ie., 
%the unique one-form $\vartheta_\ell \in \Omega ^1 ( \mathbb{T}_\ell)$ such that 
$ p _\ell ^* \vartheta_\ell = dx$, where $p_\ell: \mathbb{R} \rightarrow \TT_\ell$ is the projection $ p _\ell(x)= x +\ell\mathbb{Z}$}.
\end{lemma}

%\begin{proof} Denoting by $p_S:\tilde S\to S$ the universal covering space of $S$, the pullback by $p_S$ of the closed 1-form $\be$ is exact, so there exists $\tilde b\in C^\oo(\tilde S)$ such that $  d  \tilde b=p_S^*\be$. Since the period group $\operatorname{Per}(\beta )$ is $\ell\ZZ$, the map $\tilde b$ descends to a smooth map $b:S\to\RR/\ell\mathbb{Z}  =\TT_\ell$. Then  the pull-back of the angle element  by $b$ is $\be$, because {$p_S^*b^*\vartheta_\ell= \tilde b ^* p _\ell ^* \vartheta_\ell = \tilde b ^* dx = d \tilde b=p_S^*\be$}. \end{proof}

In other words $\be$ is the {\it logarithmic derivative} of $b$, also denoted by $\be=\de b$, and $b$ is called a {\it Cartan developing} of $\be$.

\medskip

From now on we will assume that the period group of the closed 1-form $ \beta $ is $\ell\ZZ$,
and that $\be$ has no zeros. 
This forces the Cartan developing $b:S\to \TT_\ell$ to be a submersion, hence a fibration (by compactness),
and the surface $S$ to be {of genus one}.
Notice that the 1-form $\be$ cannot be exact.
Consequently, every element $(\Si,\be_\Si)$ of the manifold of vortex sheets
$\Gr_a^{S,\be}$ has similar properties: the oriented surface $\Si\subseteq\RR^3$ is {of genus one}, it is endowed with a closed non-exact 1-form $\be_\Si$  with period group $\ell\ZZ$, without zeros, and equal to the logarithmic derivative of a fibration $b_\Si:\Si\to\TT_\ell$.
Its fibers 
$$
C_z:=b_\Si^{-1}(z),\quad z\in\TT_\ell,
$$ 
are closed curves on $\Si$, called the {\it vortex lines} of the vortex sheet $\Si$.
They integrate the kernel of $ \beta _ \Sigma $,
thus they don't depend on the choice of the Cartan developing $b_\Si$ of $\be_\Si$.

\paragraph{Frames, orientations, and curvatures for vortex lines.} The Darboux frame of the vortex line $C_z$ on the surface $\Si$
is the orthonormal frame $\{T,n_g,n\}$ with $T$ a unit tangent vector to $C_z$
(especially $\be_\Si(T)=0$),
$n$ the unit normal vector field to $\Si$, and ${n_g=n\x T}$ the  geodesic normal vector field (tangent to $\Si$). The orthonormal frame $\{T,n_g\}$  is positively oriented on $ \Sigma $:
$\mu_\Si(T,n_g)=1$.

The surface $\Si$ is oriented, so the choice of $n$ is fixed by the compatibility with the orientations of $\Si$ and $\RR^3$. The vortex lines $C_z$ are oriented by using $ \beta _ \Sigma $ and imposing the condition $\be_\Si(n_g)>0$, which provides the orientation given by $T=n_g\x n$. 

The vortex lines are the integral curves of the (nowhere vanishing) vorticity density vector field  $\ga_\Si$ with $ i _{\ga_\Si}\mu_\Si=\be_\Si$.
The vector field $\ga_\Si$ has the same orientation as the unit tangent vector field $T$,
since $\mu_\Si(\ga_\Si,n_g)=\be_\Si(n_g)>0$.
It follows that $ \beta _ \Sigma (n_g) = \| \gamma _ \Sigma \|$
and $\ga_\Si=\be_\Si(n_g)T$.

Let $\{T,N,B\}$ denote the (orthonormal) Frenet frame of the vortex line $C_z\subseteq\RR^3$.
The mean curvature vector field of the curve $C_z\subseteq\RR^3$ is $H_{C_z}= k N$,
where $ k\geq 0 $ denotes the curvature of $C_z$.
Its decomposition in the normal frame $\{n_g,n\}$ to the vortex line $C_z$ is
\begin{equation}\label{decomp_H} 
H_{C_z}= k N=k_gn_g+k_nn,
\end{equation} 
with $k_g$ and $k_n$ denoting
the geodesic curvature and the normal curvature of the fiber $C_z$ on the surface
$\Si$. Note that $N$, $k$, and $k_n$ depend only on the fibers $C_z$ (not on their orientation given by $ \beta _ \Sigma $) and on the orientation of $ \Sigma $. However, the pointing direction of the vectors $T$, $B$, $n_g$, and the sign of $ \kappa _g$ depend on the orientation
of the fibers. Note also that the {binormal curvature reads}
\begin{equation}\label{kB_decomposition}
k B=k_gn-k_nn_g.
\end{equation}

\paragraph{Gluing.} The surface $\Si$ is a disjoint union of the fibers $C_z$ of the fibration $b_\Si$,
thus the curvature functions $k$, $k_g$, and $k_n$ on the fibers 
glue together to functions on the surface. We denote them by the same letters:
$k,k_g,k_n\in C^\oo(\Si)$.
The same gluing can be done with the unit tangent vector field $T$ to the fibers, 
as well as with $n_g=n\x T$, obtaining vector fields $T,n_g\in\X(\Si)$.
It follows that the pullback of the 1-form 
\begin{equation}\label{nusi}
\sigma _\Si:=- i _{n_g}\mu_\Si=- i _{n_g} i _n\mu\in\Om^1(\Si)
\end{equation}
to each fiber,
namely $i_{C_z}^* \sigma _\Si$, is the volume form $\mu_{C_z}$ on the fiber, $z\in\TT_\ell$. The following holds true:
\begin{equation}\label{later}
\sigma _{\Si}\wedge\be_\Si=\be_\Si(n_g)\mu_\Si,
\end{equation} 
since both sides evaluated on the orthonormal frame $\{T,n_g\}$ give $\be_\Si(n_g)$.

\begin{lemma}\label{puls}
The smooth function $k_g\in C^\oo(\Si)$ is the divergence of $-n_g\in\X(\Si)$ (with respect to $\mu_\Si$), thus $\int_\Si k_g\mu_\Si=0$.
\end{lemma}
\begin{proof}
Let $\nabla$ denote the   Levi-Civita connection on $\Si$.
We compute
\begin{align*}
\div n_g=(\pounds_{n_g}\mu_\Si)(T,n_g)
{=}-( d \sigma _{\Si})(T,n_g)
= \sigma _{\Si}([T,n_g])=T\cdot [T,n_g]
=-\nabla_TT\cdot n_g
=-k_g,
\end{align*}
using the identities $\nabla_TT=k_gn_g$ and $T^\flat= \sigma _\Si$.
\end{proof}

\subsection{Hamiltonian function {and Hamilton's equations}}

On the symplectic manifold $\Gr_a^{S,\be}$ we consider the Hamiltonian function suggested in \cite{khesin} and given by the total length of the vortex lines:
\begin{equation}\label{hama}
h(\Si,\be_\Si)
=\int_{\TT_\ell} \Vol (C_z)\, \vartheta_\ell\,,
\end{equation}
where $ \Vol (C_z)=\int_{ C_z}\mu_{ C_z}$, $z\in\TT_\ell$, denotes the length of the vortex line and $\vartheta_\ell$ is the angular form defined in Lemma \ref{lemma_31}.
By the Fubini theorem for the fibration $b_\Si$:
\begin{equation}\label{fubini}
\int_{\TT_\ell} \left( \int_{C_z}\al \right) \vartheta_\ell=\int_\Si\al\wedge\be_\Si, \text{ for all }\al\in\Om^1(\Si),
\end{equation}
we can write the Hamiltonian as 
\begin{equation}\label{hama2}
h(\Si,\be_\Si)
=\int_{\TT_\ell} \left( \int_{C_z} \mu_{C_z} \right)  \vartheta_\ell =\int_ \Sigma \sigma  _ \Sigma \wedge b_ \Sigma ^* \vartheta_\ell = \int_ \Sigma \sigma  _ \Sigma \wedge \beta _ \Sigma\stackrel{\eqref{nusi}}{=} \int _ \Sigma \beta _{ \Sigma }(n_g) \mu _ \Sigma,
\end{equation}
since $\mu_{C_z}=i_{C_z} ^* \sigma  _ \Sigma$. The last equality 
 shows how $h$ differs from the area of $ \Sigma $.

\begin{lemma}\label{didi}
In the description \eqref{split_a} of the tangent space to the coadjoint orbit, namely,
\begin{equation}\label{Oo}
T_{(\Si,\beta _\Si)} \operatorname{Gr}^{S,\beta }_a=C_0^\oo(\Si) \times  d  C^\oo(\Si),
\end{equation}
the differential of the Hamiltonian function $h$ is given by
\[
( d_{(\Si,\be_\Si)}h)(\rh, d\la)= - \int_ \Sigma \big( \rh  k_n\beta _ \Sigma(n_g) + k_g \lambda \big) \mu _ \Sigma,
\]
with $k_g$ and $k_n$ the geodesic and normal curvatures of the vortex lines on $ \Sigma $.
\end{lemma}

\begin{proof}
By the transitivity result in Lemma \ref{toto},
it is enough to check the identity on infinitesimal generators of $u\in\X_{\vol}(\RR^3)$:
\[
\ze^{\Gr}_u(\Si,\be_\Si)=\big(u|_\Si\cdot n, d \be_\Si(u|_\Si^\top)\big)
=:(\rh, d \la).
\]
Let $\phi^u_t$ denote the flow of $u$. We notice that
$b_{\phi^u_t(\Si)}^{-1}(z)=\phi_t^u(b_\Si^{-1}(z))$ for all $z\in\TT_\ell$.
Then
\begin{align*}
( d _{(\Si,\be_\Si)}h)(\ze^{\Gr}_u)
&=\frac{d}{dt}\Big|_0 h\big(\phi^u_t(\Si,\be_\Si)\big)=\frac{d}{dt}\Big|_0h\big(\phi^u_t(\Si),(\phi^u_t)_*\be_\Si\big)\\
&=\frac{d}{dt}\Big|_0\int_{\TT_\ell}\Vol(b_{\phi^u_t(\Si)}^{-1}(z))\vartheta_\ell
=-\int_{\TT_\ell}\left(\int_{C_z}(u\cdot H_{C_z})\mu_{C_z}\right)\vartheta_\ell\\
&= - \int_ \Sigma \big( (u\cdot n)  k_n\beta _ \Sigma(n_g) + k_g \be_\Si(u|_\Si^\top) \big) \mu _ \Sigma= - \int_ \Sigma \big( \rh k_n\beta _ \Sigma(n_g) + k_g \lambda \big) \mu _ \Sigma.
\end{align*}
At step four we use the standard fact from the literature
that the mean curvature vector field is the direction where the volume of the submanifold decreases the fastest:
\begin{align*}
\frac{d}{dt}\Big|_0\Vol(b_{\phi^u_t(\Si)}^{-1}(z))
=\frac{d}{dt}\Big|_0\Vol(\phi^u_t(C_z))
=\frac{d}{dt}\Big|_0\int_{C_z}(\phi^u_t)^*\,\mu_{\phi^u_t(C_z)}
=-\int_{C_z}(u\cdot H_{C_z})\mu_{C_z}.
\end{align*}
For the step five we used \eqref{decomp_H} and we applied the Fubini theorem 
\eqref{fubini} for the fibration $b_\Si$.
Here $\al= (u\cdot (k_gn_g+k_nn))\sigma _{\Si}\in\Om^1(\Si)$, since 
$\mu_{C_z}=i_{C_z} ^* \sigma  _ \Sigma$,
so that by \eqref{later} we get
$$
\al\wedge\be_\Si=(k_gu\cdot n_g+k_nu\cdot n)\sigma _\Si\wedge\be_\Si
=\big(k_g\be_\Si(u|_\Si^\top)+k_n(u\cdot n)\be_\Si(n_g)\big) \mu_\Si,
$$ 
using also the identity $\be_\Si(u|_\Si^\top)=(u\cdot n_g)\be_\Si(n_g)$.
\end{proof}

\begin{proposition}
The Hamiltonian vector field associated to the Hamiltonian function \eqref{hama}, 
written in the decomposition \eqref{Oo}, is
\begin{equation}\label{xham}
X_h(\Si,\be_\Si)=\big(k_g,- d (k_n\be_\Si(n_g))\big).
\end{equation}
\end{proposition}

\begin{proof}
Notice that $k_g\in C^\oo_0(\Si)$ by Lemma \ref{puls},
so that the right hand side does belong to the tangent space \eqref{Oo}.
Let $(f, d \la)$ be an arbitrary tangent vector at $(\Si,\be_\Si)$.
Then we compute with the help of Lemma \ref{didi}:
\begin{align*}
\Om\big(X_h(\Si,\be_\Si),(\rh, d \la)\big)&=( d _{(\Si,\be_\Si)}h)(\rh, d \la)
=-\int_\Si(k_g\la+\rh k_n \be_\Si (n_g))\mu_\Si.
\end{align*}
The result follows now from the expression \eqref{cano} of $\Om$.
\end{proof}

\medskip 

In terms of the vorticity density vector field  $\gamma _ \Sigma $, the Hamiltonian function and the Hamiltonian vector field read:
\[
h(\Si,\be_\Si) =\int _ \Sigma \|\ga_\Si\|\mu _ \Si
,\qquad
X_h( \Sigma , \gamma _ \Sigma ) = \big(k_g, -  d  ( k_n \| \gamma _ \Sigma \|)\big).
\]
%Note that a change of sign in $ \beta _ \Sigma $ (with the orientation of $ \Sigma $ unchanged) induces a change of sign in $ \gamma _ \Sigma $, $T$, $n_g$, $k_g$, and $B$ (while $N$, $n$, and $k_n$ are kept unchanged), which only affects the sign of the first component of $X_h$. A change of the orientation of $ \Sigma $ (with $ \beta _ \Sigma $ fixed) induces a change of sign in $ \gamma _ \Sigma $, $T$, $n$, $k_n$, and $B$,  hence only the second component of $X_h$ changes its sign.

\paragraph{Hamilton's equations in terms of binormal curvature.} 
From \eqref{kB_decomposition} the Hamiltonian vector field can be written in terms of the binormal curvature $kB$ as
\[
X_h( \Sigma , \beta _ \Sigma ) = \big( (kB\cdot n) n,  d ( ( kB\cdot n_g) \beta _ \Sigma (n_g))\big)= \big( kB^{\perp_ \Sigma },  d  ( \beta _ \Sigma ( kB ^{\top_ \Sigma }))\big),
\]
with $kB^{\perp_ \Sigma }$ and $k B^{\top_ \Sigma }$ the components of $kB$ normal and tangent to the surface $ \Sigma $. 
This expression of the Hamiltonian vector field $X_h$ leads to the system \eqref{Ham_PDE}, equivalently written as
\[
\dot \Sigma = kB ^{\perp_\Sigma} , \qquad \dot \beta _ \Sigma = \pounds_{(kB^{\top_\Sigma})}  \beta _ \Sigma.
\]
Thus the deformation of the surface $ \Sigma $ is governed by the normal component of the binormal curvature of the vortex lines (but not by the magnitude of their strength $\beta _ \Sigma (n_g)=\| \gamma _ \Sigma \|$). The deformation of the vorticity density 1-form $ \beta _ \Sigma $ is governed by the variations of the volume $\beta _ \Sigma ( kB ^{\top_ \Sigma })= (\gamma _ \Sigma \times kB) \cdot n$ which involves both the tangential component of the binormal curvature of the vortex lines and their vorticity strength.

%%%%%%%%

\subsection{Invariance and surfaces of revolution}\label{33}

\paragraph{$\SE(3)$-invariance and momentum map.} 
The Hamiltonian $h$  on the decorated Grassmannian
is invariant under the restriction of the
$\Diff_{\rm vol}(\RR^3)$ action \eqref{act} to the special Euclidean group 
$\SE(3)$, whose action  on $ \mathbb{R} ^3 $ is simply denoted 
$x \mapsto \Phi_g(x)= Ax + a$, $g= (A,a) \in \SE(3)$. The associated SE(3)-momentum map is determined as follows.
A Lie algebra element, $ \xi =  ( \omega , v) \in \mathfrak{se}(3)$, gives rise to the divergence free vector field $u(x) = \omega \times x + v$ with vector potential $\chi(x)= \alpha^\sharp (x) = ( \omega \cdot x) x + \frac{1}{2} v \times x$. 
From \eqref{mome}, the momentum map $\mathbb{J}: \operatorname{Gr}_a^{S, \beta } \rightarrow \mathfrak{se}(3) ^*$ is
\begin{equation}\label{euclid}
\langle\mathbb{J}(\Si,\be_\Si),(\om,v)\rangle
=\frac12\int_\Si\left(||x||^2\om^\flat+(v\x x)^\flat\right)\wedge\be_\Si.
\end{equation}
Its reformulation in terms of the vortex density $\ga_\Si$,
as in \eqref{gold}, gives:
\[
\mathbb{J}( \Sigma , \ga _ \Sigma ) =  \left( \int_ \Sigma ( \gamma_ \Sigma  \cdot x) \, x \, \mu _ \Sigma ,\frac{1}{2}  \int_ \Sigma (x \times \gamma _ \Sigma)  \mu _ \Sigma  \right).
\]  
From Noether theorem it follows that the momentum map $\mathbb{J}$ 
is preserved along any solution $( \Sigma (t), \beta _ \Sigma (t))$ of Hamilton's equations.

\paragraph{Flux homomorphism and isotropy groups.} The {\it flux homomorphism} {of a closed 1-form $\be$ with $\operatorname{Per}( \beta)= \ell \mathbb{Z}$, is defined by}
\begin{equation}\label{definition_c} 
c_\beta : \operatorname{Diff}(S, \beta ) \rightarrow \TT_\ell,\quad c_ \beta (\varphi )
:=(b \circ \varphi ^{-1}) /b,
\end{equation} 
where $b$ is a Cartan developing for $ \beta $. Because $\be= \de b $ is preserved by $\varphi $, 
the map $(b \circ \varphi^{-1} ) /b:S \rightarrow \TT_\ell$ is  constant and hence $c_\be$ is well-defined.
Also, it does not depend on the choice of the Cartan developing $b:S\to \TT_\ell$, since $c_\be$ is defined up to multiplication of $b$ by a constant in $\TT_\ell$.
The derivative of $c_ \beta $ at the identity is the Lie algebra homomorphism 
\begin{equation}\label{Lie_algebra_hom} 
 d  _e c_ \beta : \mathfrak{X} (S, \beta )\rightarrow \mathbb{R},\quad  d  _e c_ \beta (v)=-\beta (v).
\end{equation}

Let us consider a diffeomorphism $\ph\in\Diff_{\rm vol}(\RR^3)$ in the isotropy group of $(\Si,\be_\Si)$,
\ie~$\ph|_\Si\in\Diff(\Si)$ and $\ph^*\be_\Si=\be_\Si$.
Since  $\de b_\Si=\be_\Si=\ph|_\Si^*\de b_\Si=\de(b_\Si\o\ph|_\Si)$ it follows that $\ph|_\Si$ satisfies $b_\Si\o\ph|_\Si=cb_\Si$, for the constant 
$c=c_{\be_\Si}(\ph|_\Si^{-1})\in\TT_\ell$.
Thus  $\ph|_\Si$ is an automorphism of the fibration, that covers a rigid rotation of the base $\TT_\ell$,
namely $z\mapsto cz$.

\paragraph{Invariant submanifolds of circle-invariant vortex sheets.} 
We consider  the subgroup $\TT_{2\pi}$ of rotations around a given axis in $\RR^3$.
The subset of $\TT_{2\pi}$-invariant elements of the decorated Grassmannian $\Gr_a^{S,\be}$,
\begin{equation}\label{H_fixed_set} 
\mathcal{R}:= \{( \Sigma , \beta _ \Sigma )\in\Gr_a^{S,\be} \mid \Phi_ g( \Sigma )= \Sigma , \; \Phi ^* _g \beta _ \Sigma = \beta _ \Sigma, \text{ for all $g \in \TT_{2 \pi }$}\},
\end{equation} 
consists of surfaces of revolution around the axis with circle-invariant fibration.
It means that for each $(\Si,\be_\Si)\in\R$, the isotropy subgroup
contains the subgroup of rotations $\TT_{2\pi}$.
In particular the fibration $b_\Si$ satisfies $\Phi_g(C_z)= C_{c z}$, for all $z \in \TT_\ell$ and all $ g \in \TT_{2\pi}$. As above, the  constant $c$ depends on $g$ through the flux homomorphism $c_{\be_\Si}$.

Since both the symplectic form $\Om$ on $\Gr_a^{S,\be}$ and the Hamiltonian function $h$ are $\TT_{2\pi}$-invariant, we have the following general result.

\begin{lemma}\label{invariant_submanifolds} The Hamiltonian vector field $X_h$ is tangent to $\R$ and hence to each of its connected components $\R_{m,n}$, which are thus preserved by the Hamiltonian dynamics.
\end{lemma}

\paragraph{Description of circle-invariant vortex sheets.} By rotating a closed plane curve $\Ga$, parametrized by $(\xi(\rho),\et(\rho))$,
with $\xi ( \rho  ) >0$, around the vertical axis, we obtain the parametrization of
a surface of revolution
\begin{equation}\label{param}
\Si:\;(\xi(\rho)\cos\th,\xi(\rho)\sin\th,\et(\rho)).
\end{equation}
Any $\TT_{2\pi}$-invariant closed 1-form $\be_\Si$ on the surface of revolution
is of the form 
\begin{equation}\label{betsi} 
\beta _ \Sigma = \zeta _ \rho  d \rho  + c d \theta .
\end{equation}
Since $\Per(\be_\Si)=\ell\ZZ$, the two terms must have period groups
$\Per(\ze_\rh d\rh)=m\ell\ZZ$ and $\Per(cd\th)=n\ell\ZZ$, with $m,n$ coprime natural numbers. In particular the constant must be $c=\frac{n\ell}{2\pi}$
and the real valued function $\ze$ satisfies the condition $\ze(\rho+2\pi)=\ze(\rho)+m\ell$.
The fibration projection $b_\Si:\Si\to\TT_\ell$ (a Cartan developing of $\be_\Si$) takes the form $b_\Si(\rho,\th)=p_\ell\left(\ze(\rh)+\tfrac{n\ell}{2\pi}\th\right)$,
for the canonical projection $p_{\ell}:\RR\to\TT_{\ell}$.

We denote by $\R_{m,n}$, $m,n$ coprime, the subset of all circle-invariant vortex sheets 
with vorticity density of the form $\be_\Si=\ze_\rh d\rh+\frac{n\ell}{2\pi}d\th$
with $\Per(\ze_\rh d\rh)=m\ell\ZZ$.
We show that $\R_{m,n}$ is connected.
Any two surfaces of revolution can be continuously deformed one into the other,
so let us consider two vortex sheets in $\R_{m,n}$ supported on the same
surface: $(\Si,\be_0)$ and $(\Si,\be_1)$. Then 
$\be_t:=(1-t)\be_0+t\be_1=\left((1-t)\ze_\rh^0+t\ze_\rh^1\right)d\rh+\frac{n\ell}{2\pi}d\th$ is a path from $\be_0$ to $\be_1$ 
consisting of vorticity densities supported on $\Si$ with $(\Si,\be_t)\in\R_{m,n}$.
%Each $\R_{m,n}$ consists of a union of connected components of $\R$ and
We conclude that $\R$ is a union of connected components:
\[
\R=
\bigcup_{m,n\text{ coprime }}\R_{m,n}.
\]
The {connected component} $\R_{1,0}$ contains surfaces of revolution fibered by parallel circles, where $ \beta _ \Sigma = \zeta _ \rho  d \rho$ with $\Per(\ze_\rh d\rh)=\ell\ZZ$.
The {connected component} $\R_{0,1}$ contains surfaces of revolution with meridian-like fibrations and $ \beta _ \Sigma $ of the form $ \beta _ \Sigma = \zeta _ \rho  d \rho  + \frac{\ell}{2\pi}d \theta $, with exact 1-form $\ze_\rh d\rh$. 

\begin{remark}\label{enne}
\rm
Note that on $ \mathcal{R} _{m,n}$ with $n\neq 0$ (hence $c\ne 0$), the vortex lines have only one associated 1-form $ \beta _ \Sigma $ (up to sign), because its period group is fixed: $\ell\ZZ$. Thus the dynamics of the vortex lines $C_z\subseteq\Si$ determines the dynamics of the vorticity density 1-form $\be_\Si$. This is not the case for $ \mathcal{R} _{1,0}$ whose vortex lines stay parallel, with vorticity density changing according to $ \zeta _ \rho  $.
\end{remark}
%%%

\subsection{Looking for stationary points}

In this section we show there are no stationary points of the Hamilton equations for vortex sheet dynamics in the $\TT_{2\pi}$-invariant subset $\R$.
We conjecture that there are no stationary points on the whole 
decorated Grassmannian $\Gr_a^{S,\be}$. Notice that also the vortex filament equation 
(on the knot space in $\RR^3$) doesn't possess stationary points.

\begin{lemma}\label{cabe}
The vortex sheet $(\Si,\be_\Si)$ is a stationary point of the Hamiltonian vector field $X_h$ if and only if all the vortex lines are geodesics of the surface $\Si$ and  the product $k\be_{\Si}(B)$ (with $k$ the curvature of the fiber and $B$ its binormal vector field) is constant on $\Si$.
\end{lemma}

\begin{proof}
A stationary point of the vector field $X_h$  in \eqref{xham} is characterized by $k_g=0$ and 
$ d (k_n\be_\Si(n_g))=0$ for the fibers of $b_\Si$. 
The first condition ensures that the fibers are geodesics, so $k_n=\pm k$ and {$n_g=\mp B$} the unit binormal vector field.
The second condition means that $k_n\be_\Si(n_g)={- k\be_\Si(B)}$ is a constant function on $\Si$.
\end{proof}

\begin{example}\label{Example_meridians}
\rm
Let $\Si$ be the torus in $\RR^3$ obtained by rotating a circle of radius $r$ around the vertical axis,
\[
((R+r\cos\rho)\cos\th,
(R+r\cos\rho)\sin\th,r\sin\rho).
\] 
We assign an orientation so that the unit normal vector field of the surface points outwards:
$n=(r\cos\rho\cos\th,r\cos\rho\sin\th,r\sin\rho)$.
We endow $\Si$ with the 1-form  $\be_\Si=\tfrac{\ell}{2\pi} d \th$: closed, without zeros, and with period group $\ell\ZZ$, thus $(\Si,\be_\Si)\in\R_{0,1}$.
Its Cartan developing $b_\Si$ is the geodesic fibration 
by meridian circles.

We show that the second condition in Lemma \ref{cabe} is not satisfied.
The imposed condition $\be_\Si(n_g)>0$ ensures that $n_g=\frac{1}{R+r\cos\rho}\pa_\th$, so $T=\frac{1}{r}\pa_\rh$.
The unit normal vector field $N$ of the meridian circle points inwards,
so $N=-n$. Thus the unit binormal vector field satisfies $B=n_g$.
The geodesic curvature $k_g$ of the meridian circle vanishes and the normal curvature $k_n=-k=-\tfrac{1}{r}$ is constant.
Still, the function $k\be_\Si(B)=-\tfrac{\ell} {2\pi r(R+r\cos\rho)}$ is not constant, hence
$(\Si,\be_\Si)$ is not a stationary point of the Hamiltonian vector field $X_h$.
\end{example}

More generally, let us see what happens for $\mathbb{T}_{2 \pi }$-invariant vortex sheets $(\Si, \be_\Si)\in\R$,
parametrized as in \eqref{param}. One computes that it induces a geodesic fibration $b_\Si$
only if the 1-form is
\[
 \beta _ \Sigma = -\frac{c\ka r( \rho  )}{ \xi ( \rho  )\sqrt{  \xi ( \rho  )^2 -\ka^2}}d\rh+cd\th,
\]
where $r ^2 = \xi _\rh ^2 + \et_\rh^2 $,
and $\ka$ is a constant.
A straightforward computation yields:
\begin{equation*}
k_n \beta _ \Sigma (n_g)=  \frac{c}{ r ^2\xi \sqrt{ r ^2 + \xi^2 \ze _\rh ^2  /c^2 }} \big( \eta _{\rh\rh} \xi _\rh- \eta _\rh \xi _{\rh\rh} + \xi \eta _\rh  \zeta _\rh ^2/c^2 \big),
\end{equation*}
which cannot be constant for our $\ze_\rh=-\frac{c\ka r}{ \xi \sqrt{  \xi ^2 -\ka^2}}$ imposed by the geodesic fibration.
Once more the Lemma \ref{cabe} implies that $(\Si,\be_\Si)$ is not a stationary point of $X_h$, hence the following result.

\begin{theorem}
There are no stationary points for the Hamiltonian vector field $X_h$ 
in \eqref{xham} restricted to $\R$.
\end{theorem}

\noindent{\bf Conjecture}: $X_h$ has no stationary points on the whole $\Gr_a^{S,\be}$.

%%%

\subsection{Surfaces of revolution fibered by parallel circles}\label{10}

In this section we determine the expression of the Hamilton equations as partial differential equations for the parametrization of the vortex sheet. This is important for further analytical and numerical studies of these equations. We focus on the invariant submanifold $\R_{1,0}$ of vortex sheets fibered by parallel circles, while the corresponding forms on the other connected components $ \mathcal{R} _{m,n}$ will be treated in a future work.
The component $\R_{1,0}$ is the only one in which the vorticity density
is not fully determined by the fibration, see Remark \ref{enne}.

Let $\Si$ denote the surface obtained by rotating the closed plane curve $\Ga$, parametrized by $(\xi(\rho),\et(\rho))$, as in \eqref{param}.
We endow it with the orientation inducing the unit normal vector field pointing inwards:
$$
n=\frac{1}{\sqrt{\xi_\rh^2+\et_\rh^2}}(-\et_\rh\cos\th,-\et_\rh\sin\th,\xi_\rh).
$$
The induced metric on $\Si$ is 
$(\xi_\rh^2+\et_\rh^2)d\rh^2+\xi^2d\th^2$ with orientation compatible
volume form 
$$\mu_\Si=i_n\mu=\xi\sqrt{\xi_\rh^2+\et_\rh^2}d\rh \wedge d\th.$$
Let us consider the fibration of the surface of revolution $\Si$ by its parallel circles
(the coordinate lines $\rh=\ $constant), which correspond to the connected component $ \mathcal{R}_{1,0}$.
This means that the fibration projection $b_\Si$ depends only on the parameter $\rh$, hence $\be_\Si=\ze_\rh d\rh$ with $\ze_\rh>0$.

The Darboux frame $\{T,n_g,n\}$ that satisfies $\be_\Si(n_g)>0$ has 
\begin{equation*}\label{enge}
n_g=\frac{1}{\sqrt{\xi_\rh^2+\et_\rh^2}}(\xi_\rh\cos\th,\xi_\rh\sin\th,\et_\rh)=\frac{1}{\sqrt{\xi_\rh^2+\et_\rh^2}}\pa_\rh
\end{equation*}
and unit tangent vector field
$T=(\sin\th,-\cos\th,0)=-\frac{1}{\xi}\pa_\th$.
Consequently, the Frenet frame $\{T,N,B\}$ of a parallel circle has $N=-(\cos\th,\sin\th,0)$, hence $B=T\x N=(0,0,-1)$.

Now the normal curvature $k_n$ and the geodesic curvature $k_g$ can be computed from the identity $kN=k_nn+k_gn_g$, where $k=\frac{1}{\xi}$
is the curvature of a parallel circle. We obtain:
\begin{equation}\label{knkg}
k_n=\frac{\et_\rh}{\xi\sqrt{\xi_\rh^2+\et_\rh^2}} \quad\quad
k_g=-\frac{\xi_\rh}{\xi \sqrt{\xi_\rh^2+\et_\rh^2}} .
\end{equation}

\medskip

The Hamiltonian vector field on the manifold of vortex sheets
$\Gr_a^{S,\be}$, restricted to the component $ \mathcal{R} _{1,0}$, reads
\begin{equation}\label{X_h_fine}
%\begin{aligned} 
X_h(\Si,\be_\Si)=\big(k_g,- d (k_n\be_\Si(n_g))\big)=\left(-\frac{\xi_\rh}{\xi \sqrt{\xi_\rh^2+\et_\rh^2}},
%&=\left(\frac{\xi_\rh}{\xi(\xi_\rh^2+\et_\rh^2)}(\et_\rh\cos\th,\et_\rh\sin\th,-\xi_\rh),
- d \left(\frac{\et_\rh\ze_\rh}{\xi(\xi_\rh^2+\et_\rh^2)}\right)\right).
%\end{aligned}
\end{equation}
This formula, {not depending on the variable $\th$, 
ensures that $ \mathcal{R} _{1,0}$ is preserved by the Hamiltonian dynamics,} consistently with Lemma \ref{invariant_submanifolds}.

The Hamiltonian dynamics above induces a dynamics 
$\Ga(t)$ on the space of oriented closed curves in the plane: $\Ga_t=k_g n_\Ga$,
where $n_\Ga$ denotes the unit normal vector field to the plane curve $\Ga$.
Let $\Gr^{S^1}(\RR^2)$ denote the nonlinear Grassmannian of oriented
closed plane curves. The euclidean metric on $\RR^2$ defines in a natural way 
a principal connection on the principal bundle $\Emb(S^1,\RR^2)\to\Gr^{S^1}(\RR^2)$
with structure group $\Diff_+(S^1)$, as in \eqref{conn}.
The horizontal lift of $\Ga(t)$ is the curve $f(t)=(\xi(t),\et(t))\in\Emb(S^1,\RR^2)$
of parametrizations that satisfies
\begin{equation}\label{horiz_param}
f_t=k_g n_{\Ga}\o f,\quad \Ga=f(S^1).
\end{equation} 
The next theorem spells out the dynamics of the parametrization functions $\xi$ and $\et$,
joined with that of the function $\ze$ that describes the vorticity density.

\begin{theorem}[{Hamilton's equations on $\R_{1,0}$}]
Let $(\Si(t),\be(t))$ be a solution of the Hamilton equation starting at $(\Si,\be_\Si) \in \R_{1,0}$, and let $\Ga(t)$ denote the closed plane curve that describes the surface of revolution $\Si(t)$.
The parametrization $f(t, \rh)=(\xi(t,\rh),\et(t,\rh))$ of this curve given in \eqref{horiz_param}  is governed by the system of partial differential equations deduced from
 the first component of \eqref{X_h_fine}:
\begin{equation}\label{eqxy}
\xi_t=\frac{\xi_\rh\et_\rh}{\xi(\xi_\rh^2+\et_\rh^2)} \quad\quad
\et_t=-\frac{\xi_\rh^2}{\xi (\xi_\rh^2+\et_\rh^2)}.
\end{equation}
The variation of the vorticity density along the parallel circles is governed by the partial differential equation deduced from
 the second component of \eqref{X_h_fine}: 
\begin{equation}\label{eqze}
\ze_t=-\frac{\et_\rh\ze_\rh}{\xi(\xi_\rh^2+\et_\rh^2)}.
\end{equation}
It depends on a solution $(\xi,\et)$ of \eqref{eqxy}.
\end{theorem}

Note that we can't get rid of the factor $\xi_\rh^2+\et_\rh^2$ by starting with an arc-length parametrized curve $\Ga$, because the arc-length parametrization condition is not compatible with the Hamilton equations.

\paragraph{Constants of motion.}
The dynamics \eqref{eqxy}--\eqref{eqze} on $ \R_{1,0}$ 
preserves  the volume enclosed by the surface  $a=\pi\int_{\TT_{2\pi}}\xi^2\et_\rh d\rh$ and the Hamiltonian function:
\[
h(\Si,\be_\Si)
\stackrel{\eqref{hama2}}{=} \int _ \Sigma \beta _{ \Sigma }(n_g) \mu _ \Sigma
=\int_\Si \frac{\ze_\rh}{\sqrt{\xi_\rh^2+\et_\rh^2}}\mu_\Si=2\pi\int_{\TT_{2\pi}}\xi\ze_\rh d\rh.
\]
In addition, by Noether theorem, it also preserves
the momentum map \eqref{euclid}. The relevant subgroup here consists of translations in the vertical direction, which commutes with the subgroup of rotations around the vertical axis.
Setting  $u(x)=e_3$, and using the potential 1-form $\al=\frac12(e _3\x x)^\flat$, whose pullback to the surface of revolution $\Si$ is $\frac12\xi^2d\th$, the corresponding constant of motion is
\[
k(\Si,\be_\Si)=\langle\mathbb{J}(\Si,\be_\Si),(0,e_3)\rangle=\int_\Si\al\wedge\be_\Si
=\int_\Si\frac12\xi^2d\th\wedge\ze_\rh d\rh=\pi\int_{\TT_{2\pi}}\xi^2\ze_\rh d\rh.
\] 
While the Hamiltonian represents the total length of the fibers, 
 $h(\Si,\be_\Si)=\int_{\TT_{2\pi}}2\pi\xi\ze_\rh d\rh=\int_{\TT_\ell}{\rm Length}(C_z)\vartheta_\ell$, this new constant of motion represents the total area of the disks $D_z$ determined by the fibers, \ie
$k(\Si,\be_\Si)=\int_{\TT_{2\pi}}\pi\xi^2\ze_\rh d\rh=\int_{\TT_\ell}{\rm Area}(D_z)\vartheta_\ell$.

\section{Prequantization of coadjoint orbits}\label{Sec_4}

In this section we show that if the coadjoint orbit of vortex sheets $\Gr_a^{S,\be}$ with $\operatorname{Per}(\be)=\ell\ZZ$ satisfies the Onsager-Feynman condition
\begin{equation}\label{OF_condition}
a\ell \in 2\pi\mathbb{Z},
\end{equation} 
then it is prequantizable, with prequantum bundle $( \mathcal{P} , \Theta )$ explicitly constructed as a certain decorated Grassmannian quotiented
by a discrete subgroup of the circle.

The same Onsager-Feynman condition was found in \cite{goldin2} applied to coadjoint orbits of infinite vortex sheets (ribbons/tubes) in $\RR^3$ ``enclosing" a finite volume $\mathcal{V}=a$. Given such an unbounded vortex sheet $(\Si,\ga_\Si)$, this condition was written there as $\mathcal{V}\Om_{\operatorname{tot}}\in 2\pi\ZZ$ with $ \Omega _{\rm tot}$ the total vorticity of the vortex sheet, defined as the integral
\[
\Om_{\operatorname{tot}}:=\int_\Ga dl\cdot(n\x\ga_\Si)
\]
along a curve $\Ga$ crossing the ribbon/tube transversally and oriented such that the frame $\{\gamma _ \Sigma , dl, n\}$ is positively oriented in $ \mathbb{R} ^3  $. This integral does not depend on the choice of $ \Gamma $.
In our setting, using $\be_\Si= i _{\ga_\Si}\mu_\Si= i _{\ga_\Si} i _n\mu$, the total vorticity can be simply written as $\Om_{\operatorname{tot}}=\int_\Ga\be_\Si$ and is equal to the smallest period $\ell$ of the closed 1-form $\be_\Si$.

In this section $S$ is a compact connected surface of genus one, endowed with a closed 1-form $\be$
without zeros and with period group $\ell\ZZ$.
We fix a Cartan developing $b: S \rightarrow \TT_\ell$ of $\be$,
\ie~$ b ^\ast \vartheta_\ell = \beta $ as in Lemma \ref{lemma_31}, uniquely defined up to a constant factor.

\subsection{Circle bundles over manifolds of vortex sheets}

In this paragraph, we define circle bundles over $\Gr^{S,\be}$ and $\Gr_a^{S,\be}$ as preliminary steps in the construction of the prequantum bundle. The total space of the circle bundle over $\Gr^{S,\be}$ is the new type of decorated nonlinear Grassmannian defined as\begin{equation}\label{grom2}
\Gr_{\ex}^{S,\be}:=\Emb(S,\RR^3)\x_{\Diff_+(S)}\O_b,
\end{equation}
(compare with the associated bundle \eqref{grom1}). Here $\O_b$ denotes the $\Diff_+(S)$ orbit of 
the Cartan developing $b\in C^\oo(S,\TT_\ell)$ of $\be$.
The orbit $\O_b$ only depends on $\be=\de b$,
hence the new decorated Grassmannian doesn't depend on the choice of the Cartan developing $b$ for $\be$.
Any another Cartan developing is a constant multiple $b'=zb$ of $b$,
and by the surjectivity of $c_\be$ it can be written as $b'=c_\be(\ph)b=b\o\ph^{-1}\in\O_b$ for some $\ph\in\Diff_+(S,\be)$.

The map 
\begin{equation}\label{fal}
[f,b']\mapsto \left(f(S),f_*b'\right)=(\Si,b_\Si),\quad f\in\Emb(S,\RR^3),\ b'\in\O_b,
\end{equation} 
makes the identification of \eqref{grom2} with 
\begin{align}\label{quot}
\operatorname{Gr}^{S, \beta }_{\rm ex}
&= \left \{(\Si, b_\Si) \in \operatorname{Gr}^S(\RR^3) \times C^\infty(\Si,\TT_\ell):  \exists\;\Psi \in \operatorname{Diff}(S,\Si),\; \Psi ^\ast b_\Si =  b \right\}.
\end{align}

The decorated Grassmannian $\operatorname{Gr}^{S, \beta }_{\rm ex}$ can be seen as the base of a principal bundle 
with structure group the normal subgroup $\operatorname{Diff} _{\ex}(S,\be)  \subseteq \operatorname{Diff}_+(S,\beta)$, defined as the kernel of 
 the flux homomorphism considered in \eqref{definition_c} 
{restricted to $\Diff_+(S,\be)$}, namely
\[
\operatorname{Diff} _{\ex}(S,\be)
:=\Ker c_\be= \{\varphi \in {\operatorname{Diff}_+ (S, \beta )} : b \circ \ph=b\}.
\]
This group formally integrates the Lie subalgebra 
of $\X(S,\be)$:
\[
\mathfrak{X}  _{\rm ex}(S, \beta )
:=\{ v \in \mathfrak{X}  (S,\beta ) :\beta (v)=0\}.
\]
Since the homomorphism $c_ \beta $ is surjective, the quotient group $\Diff_+(S,\be)/\Diff_{\ex}(S,\be)$ is isomorphic to the circle $\TT_\ell$.

To summarize, we formally get three principal bundles with structure groups $\Diff_+(S,\be)$, 
$\Diff_{\ex}(S,\be)$
and $\TT_\ell\simeq \Diff_+(S,\be)/\Diff_{\ex}(S,\be)$ as illustrated in the following diagram
\begin{equation}\label{diagram1}
\begin{xy}
\xymatrix{
\Emb(S,\RR^3)\ar[rrr]^{\Diff_{\ex}(S,\be)}_{ \pi _{\rm ex} ^ \beta }\ar@/_1pc/[rrrd]^{\pi ^ \beta }_{\Diff_+(S,\be)}& & & \Gr_{\ex}^{S,\be}\ar[d]^{\TT_\ell} _{ \bar\Pi }\\
& & &\Gr^{S,\be}\\
}
\end{xy}
\end{equation} 
with the principal bundle projections 
\begin{equation}\label{pai}
\pi^\be_{\ex}(f)=(f(S),f_*b)\quad\text{and}\quad \bar\Pi(\Si, b_\Si) = (\Si, \de b_\Si).
\end{equation}
The commutativity of the diagram, $\bar\Pi\o\pi_{\ex}^\be=\pi^\be$,
follows from $\de f_*b=f_*\de b=f_*\be$. The principal $\TT_\ell$ action  on $\operatorname{Gr}^{S, \beta }_{ \rm ex}$ is reminiscent from the principal $\Diff_+(S,\be)$
action  (by reparametrization) on the manifold of embeddings:
\begin{equation}\label{5}
c_\be(\ph)\cdot\pi_{\ex}^\be(f)=\pi_{\ex}^\be(f\o\ph), \quad\ph\in\Diff_+(S,\be),\;
f\in\Emb(S,\RR^3).
\end{equation} 
It can be written simply as 
\begin{equation}\label{T}
z\cdot(\Si,b_\Si)=(\Si,zb_\Si),\quad z\in \TT_\ell
\end{equation}
because,  by the  identity $c_\be(\ph)b=\ph_*b$, we have:
\begin{align*}
z\cdot(\Si,b_\Si)&=c_\be(\ph)\cdot\pi_{\ex}^\be(f)
\stackrel{\eqref{5}}{=}\pi_{\ex}^\be(f\o\ph){=}(f(\ph(S)),f_*\ph_*b)
=(f(S),c_\be(\ph)f_*b){=}(\Si,zb_\Si).
\end{align*}

In the constant enclosed volume setting, a similar construction yields a circle bundle over $\Gr_{a, \ex}^{S,\be}$, obtained as the associated bundle $\Gr_{a, \ex}^{S,\be}:=\Emb_a(S,\RR^3)\x_{\Diff_+(S)}\O_b$, in a similar way with \eqref{grom2}. Note the identification
\[
\operatorname{Gr}^{S, \beta }_{a, \rm ex}= 
\{(\Si,b_\Si)\in\Gr_{\ex}^{S,\be}:\Si\in\Gr_a^S\}.
\]
See the left hand side of the diagram \eqref{diagram_3} for the analogue to the commuting diagram  \eqref{diagram1} in the constant enclosed volume setting.

\paragraph{Splitting of the tangent space.} There is again a  natural splitting of the tangent space to the associated bundle
$\Emb(S,\RR^3)\x_{\Diff_+(S)}\O_b$
into its horizontal and vertical parts, using the Euclidean metric on $\RR^3$.
The tangent space to the $\Diff_+(S)$ orbit $\O_b$ of $b:S\to\TT_\ell$,
is  $\{ i _v\be:v\in\X(S)\}=\beta( \mathfrak{X}  (S))
=C^\oo(S)$, because the 1-form $\be$ has no zeros.
The analogue of the  decomposition \eqref{split} is
\begin{equation}\label{split2}
T_{(\Si,b _\Si)} \Gr_{\ex}^{S,\be}
=T_\Si\Gr^S\times \beta _\Si( \mathfrak{X}  (\Si))
=C^\oo(\Si) \times C^\infty(\Si).
\end{equation}
The bundle projection $\pi_{\ex}^\be$  being $\Diff(\RR^3)$-equivariant, the tangent map  can be written, similarly to \eqref{T_pi_beta}, as
\begin{equation}\label{infinitesimal_generators_ex} 
T_f\pi^ \beta _{\rm ex} (u \circ f)= \left( u|_\Si\cdot n,{ i  _{u|_\Si^\top}\beta _\Si}\right),
\quad \be_\Si=f_*\be,
\end{equation}
where $u|_\Si=u|_\Si^\top+u|_\Si^\perp$ is the Euclidean orthogonal decomposition of $u\in\X(\RR^3)$ along $\Si=f(S)$.
The right hand side is the  infinitesimal generator of the $\Diff(\RR^3)$ action 
on $\operatorname{Gr}_{\ex}^{S, \beta }$.

\subsection{The prequantum bundle over the manifold of vortex sheets}\label{prequantum}

We show that the Onsager-Feynman prequantization condition $a\ell\in 2\pi\ZZ$ permits the construction of a prequantum bundle $( \mathcal{P} , \Theta )$ over $\Gr_a^{S,\be}$.

Consider on $\Emb_a(S,\RR^3)$ the 1-form given by $\theta:= \widehat{\nu\cdot\be}$, namely,
\begin{equation}\label{144}
\th(u\o f)= \int_S f^* i _u\nu\wedge\be,
\end{equation}
for all $u\in\X_{\rm vol}(\RR^3)$ (see the transitivity Lemma \ref{toto}).
We have seen that the identity $\om= d \th$ with $\omega=\widehat{\mu\cdot\be}$ given in \eqref{tinta}
follows from $\mu= d \nu$, using the hat calculus.
We have also seen that the 2-form $\om$ is $\pi^\be$ basic (see Proposition \ref{simp}).
Now we check that the 1-form $\th$ is $\pi_{\ex}^\be$ basic.
Indeed, the infinitesimal generator $\ze_v$ of $v\in\X_{\ex}(S,\be)$, namely $\ze_v(f)=Tf\o v$, annihilates $\th$:
\begin{equation}\label{gog}
\th(Tf\o v)=\int_S i _vf^*\nu\wedge\be=-\be(v)\int_Sf^*\nu=-a\be(v)=0.
\end{equation}
Thus  the $\Diff_+(S,\be)$ invariant 1-form  $\th$ descends to a $\TT_\ell$ invariant 1-form $\bar{\Th}$ on $\Gr_{a,\ex}^{S,\be}$, \ie 
$(\pi_{\ex}^\be)^*\bar\Th=\th$.
Now the calculation
$(\pi_{\ex}^\be)^*\bar\Pi^*\Om=(\pi^\be)^*\Om=\om
= d \th= d (\pi_{\ex}^\be)^*\bar\Th=(\pi_{\ex}^\be)^* d \bar\Th$
implies that $\bar\Pi^*\Om= d \bar\Th$.
We summarise these facts in the next lemma.

\begin{lemma}\label{123} 
The principal $\TT_\ell$ bundle  
$$
\bar\Pi:\operatorname{Gr}^{S, \beta }_{a, \rm ex} \longrightarrow \operatorname{Gr}^{S, \beta }_a,
\quad\bar\Pi(\Si,b_\Si)=(\Si,\de b_\Si)
$$ 
is endowed with a $\TT_\ell$-invariant 1-form $\bar\Th$
that satisfies $\bar\Pi^*\Om= d \bar\Th$.
\end{lemma} 

In the splitting 
$T_{(\Si,b _\Si)} \Gr_{a,\ex}^{S,\be}
=C_0^\oo(\Si) \times C^\infty(\Si)$ deduced from \eqref{split2},
the 1-form $\bar\Th$ is expressed as
\[
\bar\Theta _{(\Si, b_\Si)}( \rh, \lambda )
=\int_\Si \rh  i  _{n} \nu\wedge\be_\Si-\int_\Si \lambda  \nu .
\]
By the transitivity Lemma \ref{toto}, it is enough to check the formula on an infinitesimal generator 
\eqref{infinitesimal_generators_ex} for $u\in\X_{\rm vol}(\RR^3)$,
namely on  $(\rh,\la)=(u\cdot n, i _{u|_\Si^\top}\be_\Si)$.
Then by \eqref{144} we have $\bar\Th_{(\Si,b_\Si)}(\rh,\la)=\th(u\o f)=\int_\Si  i _u\nu\wedge\be_\Si
=\int_\Si \rh  i  _{n} \nu\wedge\be_\Si-\int_\Si \lambda  \nu $.

In the next lemma, we show that the 1-form  $\bar\Th$ reproduces the  infinitesimal generators of the principal $\TT_\ell$ action only up to the factor $a>0$, {hence a further step is needed to construct the prequantum bundle.}

\begin{lemma}
 For all $(\Si,b_\Si)\in\operatorname{Gr}^{S, \beta }_{a, \rm ex}$,
the infinitesimal generator $\ze_s$ of the $\TT_\ell$ action is
\begin{equation}\label{name}
\bar\Th(\ze_s(\Si,b_\Si))=as, \quad s\in\RR.
\end{equation}
\end{lemma}

\begin{proof}
Given $s\in\RR$, we choose $v\in\X(S,\be)$ with $\be(v)=-s$ and we
let $\phi_t^v\in\Diff(S,\be)$ denote its flow. 
By \eqref{Lie_algebra_hom},  the curve $z(t):=c_\be(\phi_t^v)\in\TT_\ell$ satisfies $z'(0)=-\be(v)=s$, hence the infinitesimal generator of $s$ can be written as
$
\ze_s(\Si,b_\Si)=
\tfrac{d}{dt}\big|_0 \left(z(t)\cdot(\Si,b_\Si)\right)
$.
We consider $f\in\Emb_a(S,\RR^3)$ such that  $(\Si,b_\Si)=\pi_{\ex}^\be(f)$.
Then 
\[
z(t)\cdot(\Si,b_\Si)=c_\be(\phi_t^v)\cdot\pi_{\ex}^\be(f)\stackrel{\eqref{5}}{=}
\pi_{\ex}^\be(f\o\phi_t^v).
\]
Now we use $(\pi_{\ex}^\be)^*\bar\Th=\th$ to compute
\begin{align*}
\bar\Th(\ze_s(\Si,b_\Si))
&=\bar\Th(\tfrac{d}{dt}\big|_0\pi_{\ex}^\be(f\o\phi_t^v))
=(\pi_{\ex}^\be)^*\bar\Th\big(\tfrac{d}{dt}\big|_0(f\o\phi_t^v)\big)=\th(Tf\o v)
\stackrel{\eqref{gog}}{=}-a\be(v)=as,
\end{align*}
which shows the identity \eqref{name}.
\end{proof}

\medskip 

Now the prequantization condition $k = \tfrac{a\ell}{2\pi} \in \mathbb{Z}$ comes into play: the multiplication by $a$ is a well defined surjective group homomorphism
\begin{equation}\label{ma}
m_a:\TT_\ell=\mathbb{R}/\ell\mathbb{Z}\longrightarrow \TT_{2\pi}=\mathbb{R}/2\pi\mathbb{Z},
\quad s+ \ell\ZZ\mapsto as+2\pi\ZZ.
\end{equation}
The kernel of $m_a$, denoted by $K$, is isomorphic to $\ZZ_k$, and the quotient group is $\TT_\ell/K\simeq \TT_{2\pi}$.

We define $ \mathcal{P} $ as the quotient of $\Gr_{a,\ex}^{S,\be}$ with respect to the principal $\TT_\ell$ action restricted to the subgroup $K$:
$$q:\Gr_{a,\ex}^{S,\be} \longrightarrow \P:=\Gr_{a,\ex}^{S,\be}/K.$$ 
Notice that the bundle projection  $\bar\Pi$ in Lemma \ref{123} descends to 
\begin{equation}\label{PB}
\Pi:\P\longrightarrow \operatorname{Gr}^{S, \beta }_a,
\quad\Pi\o q=\bar\Pi,
\end{equation}
a principal $\TT_{2\pi}$ bundle for the action on $\P$ that descends from the principal $\TT_\ell$ action on $\Gr_{a,\ex}^{S,\be}$:
\begin{equation}\label{1}
m_a(z)\cdot q(\Si,b_\Si):=q(z\cdot(\Si,b_\Si)),\quad z\in \TT_\ell.
\end{equation}
Here $m_a(z)$ runs through the whole  group $\TT_{2\pi}$.
The diagram below illustrates these constructions.
\begin{equation}\label{diagram_3} 
\begin{xy}
\xymatrix{
\left( \Emb_a(S,\RR^3), \theta \right) \ar[rrr]^{\Diff_{\ex}(S,\be)}_{ \pi _{\rm ex} ^ \beta }\ar@/_1pc/[rrrd]^{\pi ^ \beta }_{\Diff_+(S,\be)}& & & \left( \Gr_{a,\ex}^{S,\be}\ar[d]^{\TT_\ell} _{\bar \Pi }, \bar \Theta \right) \ar[rrr]^K_q & & & \left( \mathcal{P} = \Gr_{a,\ex}^{S,\be}/K,\Theta \right) \ar@/^1pc/[llld]_\Pi^{ \TT_{2 \pi }} \\
& & &\Gr_a^{S,\be}& & &\\
}
\end{xy}
\end{equation}

\begin{theorem} \label{pre}
If the Onsager-Feynman condition $ a\ell \in 2\pi\mathbb{Z}$ holds, the manifold of vortex sheets 
$\operatorname{Gr}_a^{S, \beta }$ endowed with the symplectic form $\Om$ is prequantizable.
A prequantum bundle is given by  the principal $\TT_{2\pi}$ bundle $\Pi:\P\to \operatorname{Gr}^{S, \beta }_a$
endowed with the connection 1-form $\Th$ that descends from $\bar\Th$, \ie $q^*\Th=\bar\Th$. {In particular, the coadjoint orbits of vortex sheets satisfying the Onsager-Feynman condition are prequantizable}.
\end{theorem}
\begin{proof} 
The 1-form $\Th$ is $\TT_{2\pi}\simeq\TT_\ell/K$ invariant because $\bar\Th$ is $\TT_\ell$-invariant.
Note that, by \eqref{1}, the infinitesimal generators 
$\ze_s$ on $\Gr_{a,\ex}^{S,\be}$ and $\ze_{as}$ on $ \mathcal{P} $ are $q$-related.
Now, since $q^*\Th=\bar\Th$, the identity \eqref{name} implies that
$\bar\Th$  reproduces the generators of the infinitesimal principal $\TT_{2\pi}$-action.
Thus the 1-form $\Th$ is a principal connection.
That its curvature  is the symplectic form $\Om$, \ie  $ d \Th=\Pi^*\Om$,
follows from Lemma \ref{123}
\end{proof}

%\begin{remark}[Prequantization of $\R$]
%{\rm
%We know that the subset $ \mathcal{R} $ consisting of $\mathbb{T}_{2 \pi }$ invariant vortex sheets, see \eqref{H_fixed_set}, is a symplectic submanifold. Thus $\Pi:\P_\R:=\Pi^{-1}(\R)\to\R$ is a prequantum bundle  with connected components 
%$$\Pi:\P_{\R_{m,n}}\to\R_{m,n},\;\text{ for $m,n$ coprime numbers}.$$
%Notice that the group $\TT_{2\pi}$ also acts on $\P$, and the subset $\P^{\TT_{2\pi}}=\P_{\R_{1,0}}$ fixed by this circle action is the prequantum bundle over the component $\R_{1,0}$ studied in Section \ref{10}.}
%\end{remark}

%%%%%%%%%%%%%%%

\section{Characters and polarizations}\label{cara}

In this section we present a polarization for the momentum $J(\Si,\be_\Si)
\in{\X_{\vol}(\RR^3)}^*$, with $(\Si,\be_\Si)\in\Gr_a^{S,\be}$,
polarization that extends the one defined in \cite{goldin2} to closed vortex sheets.
Assuming that the Onsager-Feynman condition $k = a\ell \in 2\pi\mathbb{Z}$ holds, with $\ell$ the smallest period of $\be_\Si$, we build geometrically a character for this momentum. This character can be extended from the isotropy subgroup of $(\Si, \beta  _ \Sigma )$ to the polarization subgroup.

\subsection{Polarization subgroup}

A character associated to a momentum $m\in\g^*$ is  a group homomorphism
$\chi:G_m\to \TT_{2\pi}$ defined on the isotropy subgroup associated to $m$
(for the coadjoint action),
that integrates the Lie algebra homomorphism obtained by the restriction of 
the momentum to its isotropy Lie algebra
$\g_m$, \ie~$ d _e\chi= m|_{\mathfrak{g}_m}$.  It defines a line bundle  $G\x_{\chi} \mathbb{C}\to G/G_m$ over the coadjoint orbit $\O_m$. 

A polarization group associated to a momentum $m\in\g^*$
is a Lie group $H$  such that $G_m\subset H\subset G$, whose Lie algebra $\h$ satisfies $m|_{[\h,\h]}=0$ (\ie $m|_\h:\h\to\RR$ Lie algebra homomorphism).
Moreover, if the character $\chi$ can be extended to a group homomorphism on $H$,
then we also get a polarization line bundle $G\x_{\chi} \mathbb{C}\to G/H$ 
over the configuration space.
In finite dimensions it is also asked that $\dim (G/H)=\frac12\dim (G/G_m)$, so the configuration space  is
half the dimension of the coadjoint orbit.

%%%%

Let $(\Si,\be_\Si)\in\Gr_a^{S,\be}$ be a vortex sheet $ \Sigma $ fibered by
vortex lines with vorticity density described by the closed 1-form $ \beta _ \Sigma $.
Let us consider the subgroup ${\operatorname{Diff}_{\vol}(\RR^3)}_{\Si} \subseteq \operatorname{Diff}_{\vol}(\RR^3)$ consisting of all compactly supported volume preserving diffeomorphisms of $\RR^3$ that leave $\Si$ invariant as a set. It contains the isotropy subgroup ${\operatorname{Diff}_{\vol}(\RR^3)}_{( \Sigma , \be_\Si)}$.

\begin{lemma}
The identity component $H$ of ${\operatorname{Diff}_{\vol}(\RR^3)}_{\Si}$
is a polarization group for the momentum $J(\Si,\be_\Si)$.
\end{lemma}

\begin{proof}
The corresponding Lie algebra
$\h={\mathfrak{X}  _{\vol}(\RR^3)}_\Si$, which consists of divergence free
vector fields that are tangent to $\Si$, satisfies the 
polarization condition $J(\Si,\be_\Si)|_{[\h,\h]}=0$. Indeed, for all $X_{ \alpha _1 }, X_{ \alpha _2 } \in \mathfrak{h}  $, we have
\[
\left\langle J(\Si, \be_{\Si}), [ X_{ \alpha _1 }, X_{ \alpha _2 }] \right\rangle =\int_ \Si (  i  _{ X_{ \alpha _1 }} i  _{ X_{ \alpha _2 }} \mu )\wedge\be_{\Si} =0, 
\]
using that  fact that $  i  _{ X_{ \alpha _1 }} i  _{ X_{ \alpha _2 }} \mu$ is a potential form for
the divergence free vector field $[ X_{ \alpha _1 }, X_{ \alpha _2 }]$
on $\RR^3$.
\end{proof}

%\begin{remark}[{Coadjoint orbits of vortex filaments have no polarizations}]
%\rm Let $S^1$ be the circle and let $ \Gr^{S^1}$ be the space of knots (the nonlinear Grassmannian of closed curves in $\RR^3$) endowed with the Marsden-Weinstein symplectic form (see \cite{Marsden-Weinstein}). The action of ${\operatorname{Diff}_{\vol}(\RR^3)}$ on $ \Gr^{S^1}$ is Hamiltonian with momentum map 
%\[
%J:\Gr^{S^1}\to\X_{\vol}(\RR^3)^*,\quad\left\langle J(C), X_\al \right\rangle =\int_{C} \alpha,
%\]
%expression that doesn't depend on the choice of the potential 1-form $\al$. It realizes each connected  component of the knot space $\Gr^{S^1}$ as a coadjoint orbit of the identity component of $\Diff_{\vol}(\RR^3)$ \cite{Ismagilov}\cite{Haller-Vizman}. An argument why there exists no polarizations  for these coadjoint orbits can be found in \cite{goldin2}.
%\end{remark}

%%%%%%

\subsection{A geometric construction of the character}\label{333}

Here $S$ is again a compact connected oriented surface of genus one, endowed with a closed 1-form $\be$
without zeros and with discrete period group $\ell\ZZ$.
Let us choose $(\Si,\be_\Si)\in\Gr_a^{S,\be}$ and  $b_\Si:\Si\to\TT_\ell$
 a Cartan developing for $\be_\Si$.
The Fubini theorem on fiber bundles \eqref{fubini} permits to write the momentum 
associated to it as:
\begin{equation}\label{momeo}
\langle J(\Si,\be_\Si),X_\al\rangle=\int_{\Si}\al\wedge\be_\Si =\int_{\TT_\ell}\left(\int_{b_\Si^{-1}(z)}\al\right){\vartheta_\ell} .
\end{equation}

As a first step towards the construction of the character, we integrate the
Lie algebra homomorphism $J(\Si,\be_\Si){|_{\h}}$ to a group homomorphism 
$\tilde\chi:\tilde H\to\RR$, with $\tilde H$ the universal covering group of 
the polarization group $H$, the identity component of $\Diff_{\vol}(\RR^3)_\Si$.
A point in $\tilde H$ is a homotopy class of smooth paths $\ph_t\in H$
with fixed endpoints: the identity $\ph_0$ and $\ph_1=\ph$.
The fact that the group structure on this universal cover 
can be viewed either as composition of diffeomorphisms or as path juxtaposition, 
as done in \cite[Section 10.2]{MDSa98} for the Hamiltonian group,
will be used in the proof of Lemma \ref{mcduff} below.
 
\smallskip

For each $z\in \TT_\ell$, let $C_z=b_\Si^{-1}(z)$ be the fiber over $z$.
The isotopy $\ph_t$ preserves  $\Si$,
hence the 2-chain $C_z^{\ph_t}$ swept out by $C_z$ under
the path of diffeomorphisms $\ph_t$
from time 0 to time 1 stays on $\Si$.
A smooth singular 2-chain $D_z$ in $\RR^3$ with boundary $C_z$
sweeps out a smooth singular 3-chain $D_z^{\ph_t}$ under the path of diffeomorphisms $\ph_t$ from time 0 to time 1.
In the special case when $D_z$ can be chosen inside the genus one surface $\Si$,
then $D_z^{\ph_t}$ stays inside $\Si$, as illustrated in Figure 1.
With these notations, we state the following result.

\begin{figure}
\[
\includegraphics[width=8cm]{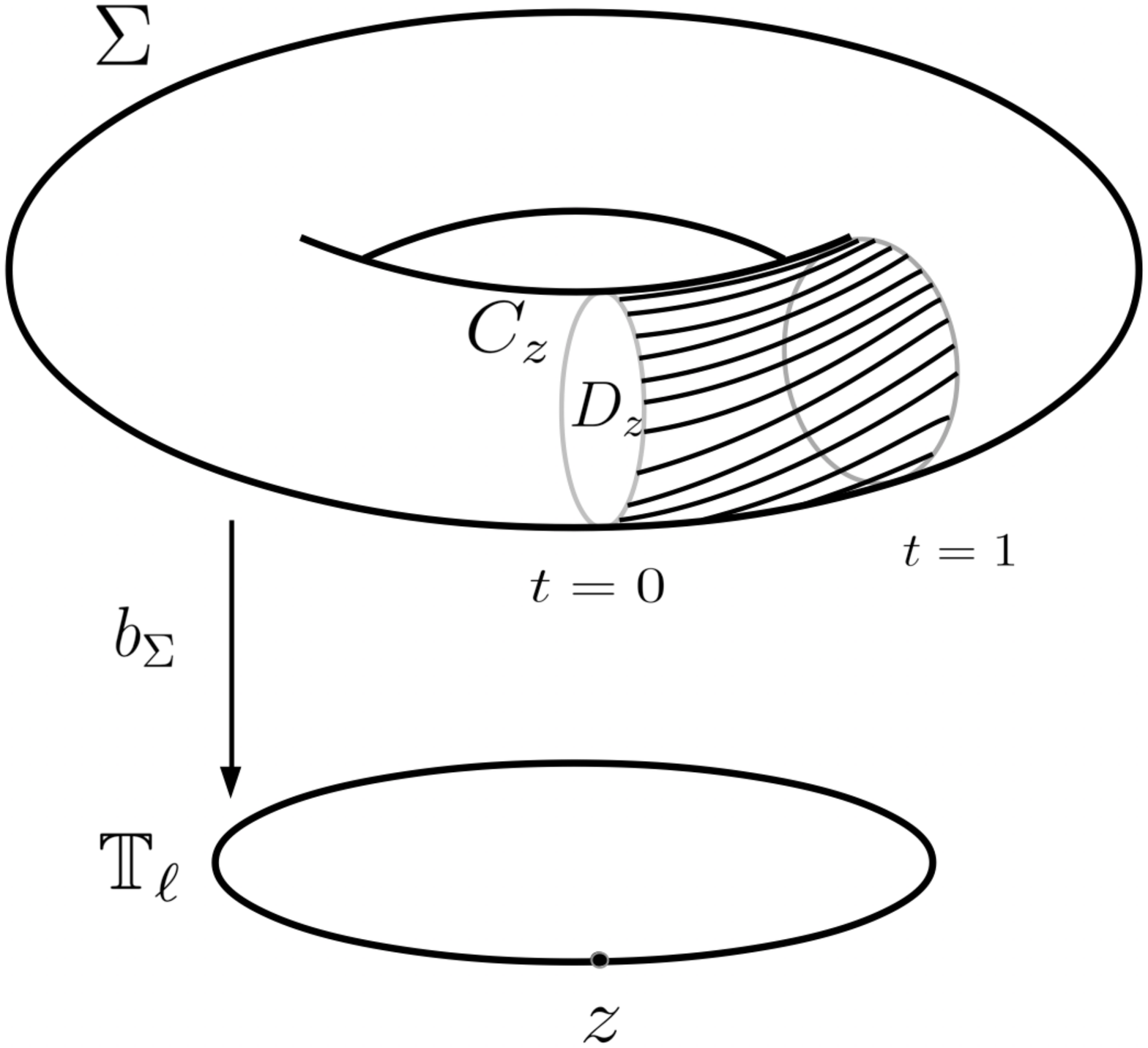}
\]
\caption{An illustration of the construction of the character
for $\varphi$ in the polarization group. For each $z\in \TT_\ell$, the volume enclosed by the hatched area represents the volume $D_z^{\ph_t}$ swept out by $ D_z$  under the path of diffeomorphisms $\ph_t$ from time $0$ to time $1$, with $\varphi_1=\varphi$.}
\end{figure}

\begin{lemma} \label{mcduff}
The map assigning to the homotopy class $\{\ph_t\}$ the total volume of $D_z^{\ph_t}$ over the base  $\TT_\ell$,
\begin{equation}\label{suflu}
\widetilde\chi:\widetilde H\to\RR,\quad \widetilde\chi(\{\ph_t\}):=\int_{\TT_\ell}\left (\int_{D_z^{\ph_t}}\mu\right) \vartheta_\ell,
\end{equation}
is a well defined group homomorphism.
\end{lemma} 

\begin{proof} The volume $\int_{D_z^{\ph_t}}\mu$ swept out by $D_z$ only depends on $C_z$ and $\ph_t$ and not on the choice of $D_z$. Indeed, for $\pa D_z=\pa\bar D_z=C_z$ and $\mu=d\nu$, we get
\[
\int_{D_z^{\ph_t}-\bar D_z^{\ph_t}}\mu
={\int_{\ph_1(D_z-\bar D_z)}\nu}-\int_{D_z-\bar D_z}\nu
=\int_{D_z-\bar D_z}(\ph_1^*\nu-\nu)=0,
\] 
because $D_z-\bar D_z$ is a cycle and $\ph_1^*\nu-\nu\in\Om^2(\RR^3)$ is  exact.
The independence of $b_\Si$ follows by the translation invariance of the volume form $\vartheta_\ell$ on $\TT_\ell$, see the proof of Lemma \ref{lemma_31}.

Next we check the independence of \eqref{suflu} on the path in the homotopy class
$\{\ph_t\}$. Let $\bar\ph_t$ be a path from the identity to $\ph$
homotopic to $\ph_t$ through the homotopy $\Ps_{s,t}\in H$ with fixed endpoints.
The boundary of the 3-chain swept out by $C_z$ under this homotopy,
$C_z^{\Ps_{s,t}}$ is $C_z^{\bar\ph_t}-C_z^{\ph_t}$ and coincides with the boundary of the 3-chain $D_z^{\bar\ph_t}-D_z^{\ph_t}$. We get
\[
\int_{D_z^{\bar\ph_t}}\mu-\int_{D_z^{\ph_t}}\mu
=\int_{C_z^{\bar\ph_t}-C_z^{\ph_t}}\nu
=\int_{\pa C_z^{\Ps_{s,t}}}\nu
=\int_{C_z^{\Ps_{s,t}}}\mu=0
\] 
by dimensional reasons: the 3-chain $C_z^{\Ps_{s,t}}$ lies on the surface $\Si$,
since each diffeomorphism $\Ps_{s,t}$ preserves $\Si$.
Let $\ph_t$ and $\ps_t$ be two paths in $H$ starting at the identity. 
We use the homotopy equivalence of the paths
\begin{align*}
\ph_t\o\ps_t\;\;\text{ and }\;\;
\et_t:=
\begin{cases}
\ph_{2t},\quad &t\in\left[0,\tfrac12\right]\\
\ph_1\o\ps_{2t-1},\quad &t\in\left[\tfrac12,1\right]
\end{cases}
\end{align*}
in the next computation, taking into account that $\ph_1^*\mu=\mu$,
\[
\tilde\chi(\{\ph_t\}\{\ps_t\})=\tilde\chi(\{\et_t\})
=\int_{\TT_\ell}\left (\int_{D_z^{\ph_t}}\mu\right) \vartheta_\ell
+\int_{\TT_\ell}\left (\int_{D_z^{\ps_t}}\ph_1^*\mu\right) \vartheta_\ell
=\tilde\chi(\{\ph_t\})+\tilde\chi(\{\ps_t\}).
\]
This ensures the homomorphism property of $\tilde\chi$.
\end{proof}

\medskip

Each path is the flow of a time dependent vector field $u_t\in\h =\X_{\vol}(\RR^3)_\Si$, namely $\tfrac{d}{dt}\ph_t=u_t\o\ph_t$.
The volume swept out by $D_z$ can also be expressed 
with the help of potential 1-forms $\al_t\in\Om^1(\RR^3)$, \ie $ i _{u_t}\mu= d \al_t$, as
\begin{equation*}
\int_{D_z^{\ph_t}}\mu=\int_0^1\left(\int_{ D_z} i _{\ph_t^*u_t}\mu\right)dt=\int_0^1\left(\int_{ D_z}\ph_t^*( i _{u_t}\mu)\right)dt
=\int_0^1\left(\int_{ C_z}\ph_t^*\al_t\right)dt.
\end{equation*}
A formula for the character in the style of \eqref{momeo} follows:
\begin{equation}\label{fine2}
\widetilde\chi(\{\ph_t\})
=\int_{\TT_\ell}\left(\int_0^1\left(\int_{ C_z}\ph_t^*\al_t\right)dt\right)\vartheta_\ell.
\end{equation}

\begin{lemma}
The group homomorphism $\widetilde\chi:\widetilde H\to\RR$ integrates the restriction to the polarization Lie algebra $\h=\X_{\vol}(\RR^3)_\Si$ 
of the momentum
$J(\Si,\be_\Si)$ associated to $(\Si,\be_\Si)\in\Gr_a^{S,\be}$.
\end{lemma}

\begin{proof}
Let $\phi_t^u$ denote the flow at time $t$ of the vector field $u\in\h$
and consider the path $\ep\mapsto \{\phi_t^u\}_{0\le t\le\ep}$ of homotopy classes in $\widetilde H$. 
Its derivative at $\ep=0$ is equal to  $u$, so  the associated Lie algebra homomorphism is
\begin{align*}
u=X_\al\mapsto\frac{d}{d\ep}\Big|_{\ep=0}\widetilde\chi( \{\phi_t^u\}_{0\le t\le\ep})
&
\stackrel{\eqref{fine2}}{=}\frac{d}{d\ep}\Big|_{\ep=0}\int_{\TT_\ell}\Big(\int_0^\ep\left(\int_{ C_z}(\phi_t^u)^*\al\right)dt\Big)\vartheta_\ell\\
&=
\int_{\TT_\ell}\left(\int_{ C_z}\al\right)\vartheta_\ell\stackrel{\eqref{momeo}}{=}\langle J(\Si,\be_\Si),X_\al\rangle,
\end{align*}
which is the required momentum.
\end{proof}

\begin{theorem}\label{fain}
Let $(\Si,\be_\Si)\in\Gr_a^{S,\be}$ and
assume that the Onsager-Feynman condition $ a\ell \in 2\pi\mathbb{Z}$ holds.
Then a  character on the polarization group $H$ is given by
$$
\chi:H\to \TT_{2\pi},\quad \chi(\ph):=
p_{2\pi}\left(\int_{\TT_\ell}\left(\int_{D_z^{\ph_t}}\mu\right)\vartheta_\ell\right),
$$ 
where $\ph_t\in H$ is any path from the identity to $\ph$.
\end{theorem}

\begin{proof}
The fundamental group $\pi_1(H)$ consists of homotopy classes of loops $\ph_t\in H$ based at the identity. We only need to check that every such class $\{\ph_t\}\in\pi_1(H)$,
is mapped by the homomorphism  $\tilde\chi$, given in \eqref{suflu}, to $2\pi\ZZ$.
The boundary of the 3-chain $D_z^{\ph_t}$ swept out by $D_z$
 under the loop of diffeomorphisms $\ph_t$
is the  $2$-chain $C_z^{\ph_t}\subseteq \RR^3$
swept out by the fiber $C_z=\pa D_z\subseteq\Si$.
Notice that  each diffeomorphism $\ph_t\in H$
preserves the surface $\Si$ and we get a smooth singular $2$-cycle $C_z^{\ph_t}\subseteq\Si$, which is also integral. Thus its homology class must be 
an integer multiple  of the fundamental class of the surface $\Si$.
The integer will not depend on $z\in \TT_\ell$ (by continuity reasons), hence 
$[C_z^{\ph_t}]=k[\Si]$ with $k\in\ZZ$.
(In the special case when the smooth singular 2-chains $D_z$ can be chosen inside $\Si$, 
the whole 3-chain $D_z^{\ph_t}$ remains inside $\Si$ and it covers it $k$ times.)

The total volume enclosed by $\Si$ is  $a=\int_\Si\nu$.
Then
\[
\widetilde\chi(\{\ph_t\})
=\int_{\TT_\ell}\left (\int_{D_z^{\ph_t}}\mu\right) \vartheta_\ell
=\int_{\TT_\ell}\left (\int_{C_z^{\ph_t}}\nu\right) \vartheta_\ell
=\int_{\TT_\ell}\left (\int_{k\Si}\nu\right) \vartheta_\ell
=\int_{\TT_\ell}(ka)\vartheta_\ell=ka\ell\in 2\pi\ZZ,
\]
by the prequantization condition $ a\ell \in 2\pi\mathbb{Z}$.
\end{proof}

\medskip
The isotropy subgroups $\Diff_{\vol}(\RR^3)_{(\Si,\be_\Si)}$ and $\Diff_{\vol}(\RR^3)_{J(\Si,\be_\Si)}$ coincide.
We have seen in \S\ref{33} that an element $\ph$ of the isotropy subgroup
maps the vortex line $C_z$ into the vortex line $C_{cz}$ with
$c=c_{\be_\Si}(\ph)$.
We show below 
a  character formula that doesn't involve isotopies.

\begin{proposition}
The restriction of the character $\chi$ from the Proposition \ref{fain} to the isotropy subgroup of $(\Si,\be_\Si)\in\Gr_a^{S,\be}$ is also given by the formula
\[
\chi(\ph)=p_{2\pi}\Big(\int_{\Si} d ^{-1}(\ph^*\nu-\nu)\wedge\be_\Si\Big)m_a\big(c_{\be_\Si}(\ph)^{-1}\big),
\]
where $p_{2\pi}:\RR\to\TT_{2\pi}$ denotes the canonical projection and  $m_a:\TT_\ell\to \TT_{2\pi}$ the multiplication \eqref{ma}.
\end{proposition}
\begin{proof}
Let $u_t\in\X_{\vol}(\RR^3)_{(\Si,\be_\Si)}$ be the time dependent vector field
associated to the path $\ph_t\in\Diff_{\vol}(\RR^3)_{(\Si,\be_\Si)}$, and let $\al_t$ be the potential 1-forms. Then $d\al_t=\pounds_{u_t}\nu-di_{u_t}\nu$.
We apply the Fubini theorem to \eqref{fine2} and we compute
\begin{align*}
\widetilde\chi(\{\ph_t\})&=\int_{\TT_\ell}\left(\int_0^1\left(\int_{ C_z}\ph_t^*\al_t\right)dt\right)\vartheta_\ell
=\int_0^1\int_\Si\ph_t^*\al_t\wedge\be_\Si dt\\
&=\int_0^1\int_\Si d^{-1}(\ph_t^*\pounds_{u_t}\nu)\wedge\be_\Si dt
-\int_0^1\int_\Si\ph_t^*i_{u_t}\nu\wedge\be_\Si dt\\
&=\int_0^1\int_\Si\frac{d}{dt} d^{-1}(\ph_t^*\nu-\nu)\wedge\be_\Si dt
-\int_0^1\int_\Si i_{u_t}\nu\wedge\be_\Si dt\\
&=\int_\Si d^{-1}(\ph_1^*\nu-\nu)\wedge\be_\Si
+a\int_0^1\be_\Si(u_t) dt,
\end{align*}
using at step four  the fact that $\ph_t$ preserves $\Si$ and $\be_\Si$,
at step five the fact that $u_t$ is tangent to $\Si$, and at the last step
the fact that $\be_\Si(u_t)$ is a constant function on $\Si$.
We know that $p_{2\pi}(\chi(\{\ph_t\}))=\chi(\ph_1)$, so the result follows from:
\begin{align*}
p_{2\pi}\Big(a\int_0^1\be_\Si(u_t) dt\Big)=m_a\Big(p_\ell\Big(\int_0^1\be_\Si(u_t) dt\Big)\Big)
=m_a\big(c_{\be_\Si}(\ph_1)^{-1}\big),
\end{align*}
knowing from \eqref{Lie_algebra_hom} that the derivative of the flux homomorphism $c_{\be_\Si}$ is $-\be_\Si$.
\end{proof}

\paragraph{Acknowledgement.} Both authors were partially supported by the LEA Franco-Roumain ``MathMode". FGB was also partially supported by the ANR project GEOMFLUID, ANR-14-CE23-0002-01. CV was also partially supported by CNCS UEFISCDI, project number  PN-III-P4-ID-PCE-2016-0778.
We would like to thank {Stefan Haller} and Boris Khesin  for very helpful comments and references.

%%%%%%%%%%%%%%%

{
\footnotesize

\bibliographystyle{new}

}

\end{document}